\documentclass[a4paper,12pt]{article}

\newenvironment{proof}{\noindent {\bf Proof:}}{\hfill $\Box$}

\newtheorem{theorem}{Theorem}
\newtheorem{lemma}{Lemma}

\newtheorem{algorithm}{Algorithm}

\newtheorem{remark}{Remark}

\usepackage{booktabs}

\usepackage{tikz}

\usepackage{algorithm}
\usepackage[noend]{algpseudocode}
\makeatletter
\def\BState{\State\hskip-\ALG@thistlm}
\makeatother
\usepackage{algorithmicx}

\usepackage{amsmath}
\usepackage{amssymb}
\usepackage{amsfonts}
\usepackage{graphicx}
\usepackage{ dsfont }

\usepackage[normalem]{ulem}
\usepackage{color}

\usepackage{color}

\usepackage{tikz}
\usepackage{lipsum}

\usepackage{url}

\newcommand{\mr}[1]{\mathrm{#1}}
\newcommand{\emhp}[1]{\emph{#1}}

\newcommand{\M}{\mathop{\mathrm{M}}}

\newcommand{\C}{\mathcal{C}}
\newcommand{\D}{\mathcal{D}}

\newcommand{\R}{\mathcal{R}}

\newcommand{\Rb}{\mathbb{R}}
\newcommand{\Cb}{\mathbb{C}}
\newcommand{\Nb}{\mathbb{N}}
\newcommand{\Ib}{\mathbb{I}}

\newcommand{\bs}{\boldsymbol}

\newcommand{\Kc}{\mathcal{K}}

\newcommand*{\herm}{^{\mathsf{H}}}

\newcommand{\Lc}{\mathcal{L}}

\newcommand{\new}[1]{{\color{black}#1}}
\def\be{\begin{equation}}
\def\ee{\end{equation}}
\newcommand{\lattice}{\mathop{\mathrm{lattice}}}

\algrenewcommand\algorithmicensure{\textbf{Output:}}

\newcommand{\ob}{{\bs h}}
\newcommand{\conjug}{{\zeta}}

\title{\bf \new{Optimal construction of Koopman eigenfunctions for prediction and control}}

\begin{document}

\author{Milan Korda$^{1,2}$, Igor Mezi{\'c}$^3$}

\footnotetext[1]{CNRS, Laboratory for Analysis and Architecture of Systems (LAAS), Toulouse, France. \texttt{korda@laas.fr}.}
\footnotetext[2]{Czech Technical University in Prague, Faculty of Electrical Engineering, Department of Control Engineering, Prague, Czech Republic.}
\footnotetext[3]{ University of California, Santa Barbara. {\tt mezic@ucsb.edu}.}

\maketitle
% LONG ABSTRACT
\begin{abstract}
 This work presents a novel data-driven framework for constructing eigenfunctions of the Koopman operator geared toward prediction and control. The method leverages the richness of the spectrum of the Koopman operator away from attractors to construct a set of  eigenfunctions such that the state (or any other observable quantity of interest) is in the span of these eigenfunctions and hence predictable in a linear fashion. The eigenfunction construction is optimization-based with no dictionary selection required. Once a predictor for the uncontrolled part of the system is obtained in this way, the incorporation of control is done through a \emph{multi-step} prediction error minimization, carried out by a simple linear least-squares regression. The predictor so obtained is in the form of a linear controlled dynamical system and can be readily applied within the Koopman model predictive control framework of~\cite{korda2018linear} to control \emph{nonlinear} dynamical systems using \emph{linear} model predictive control tools. The method is entirely data-driven and based predominantly on convex optimization. The novel eigenfunction construction method is also analyzed theoretically, proving rigorously that the family of eigenfunctions obtained is rich enough to span the space of all continuous functions. In addition, the method is extended to construct \emph{generalized eigenfunctions} that also give rise Koopman invariant subspaces and hence can be used for linear prediction.  Detailed numerical examples\footnote{Code for the numerical examples is available from \url{https://homepages.laas.fr/mkorda/Eigfuns.zip}} demonstrate the approach, both for prediction and feedback control.

%The eigenfunctions are constructed by flowing forward the values of given boundary functions defined on a non-recurrent set of initial conditions. The freedom in choosing the boundary functions and the associated eigenvalues allows for construction of a family of eigenfunctions rich enough to span the space of all continuous functions, which we rigorously prove.

 \end{abstract}

\begin{center}\small
\end{center}

 \section{Introduction}
 
The Koopman operator framework is becoming an increasingly popular tool for data-driven analysis of dynamical systems. In this framework, a nonlinear system is represented by an infinite dimensional \emph{linear} operator, thereby allowing for spectral analysis of the nonlinear system akin to the classical spectral theory of linear systems or Fourier analysis. The theoretical foundations of this approach were laid out by Koopman in~\cite{Koopman:1931} but it was not until the early 2000's that the practical potential of these methods was realized in~\cite{MezicandBanaszuk:2004} and \cite{mezic2005spectral}. The framework became especially popular with the realization that the Dynamic Mode Decomposition (DMD) algorithm~\cite{schmid:2010} developed in fluid mechanics constructs an approximation of the Koopman operator, thereby allowing for theoretical analysis and extensions of the algorithm (e.g., \cite{williams2015data,arbabi2017ergodic,korda2018convergence}). This has spurred an array of applications in fluid mechanics~\cite{rowley2009}, power grids~\cite{Susukietal:2016}, neurodynamics~\cite{brunton2016extracting}, energy efficiency~\cite{GeorgescuandMezic:2015},  or molecular physics~\cite{wuetal:2016}, to name just a few.

Besides descriptive analysis of nonlinear systems, the Koopman operator approach was also utilized to develop systematic frameworks for control~\cite{korda2018linear} (with earlier attempts in, e.g., \cite{brunton_dmd_control,Koopman_cont_extend}), state estimation~\cite{surana_estim,surana2017koopman} and system identification~\cite{mauroy2019koopman} of nonlinear systems. All these works rely crucially on the concept of \emph{embedding} (or lifting) of the original state-space to a higher dimensional space where the dynamics can be accurately predicted by a linear system. In order for such prediction to be accurate over an extended time period, the embedding mapping must span an \emph{invariant subspace} of the Koopman operator, i.e., the embedding mapping must consist of the (generalized) eigenfunctions of the Koopman operator (or linear combinations thereof).

It is therefore of paramout importance to construct accurate approximations of the Koopman eigenfunctions. The leading data-driven algorithms are either based on the Dynamic Mode Decomposition (e.g., \cite{schmid:2010,williams2015data}) or the Generalized Laplace Averages (GLA) algorithm~\cite{mohr2014construction}.  The DMD-type methods can be seen as finite section operator approximation methods, which do not exploit the particular Koopman operator structure and enjoy only weak spectral convergence guarantees~\cite{korda2018convergence}. On the other hand, the GLA method does exploit the Koopman operator structure and ergodic theory and comes with spectral convergence guarantees, but suffers from numerical instabilities for eigenvalues that do not lie on the unit circle (discrete time) or the imaginary axis (continuous time). Among the plethora of more recently introduced variations of the (extended) dynamic mode decomposition algorithm, let us mention the variational approach~\cite{wu2017variational}, the sparsity-based method~\cite{kaiser2017data} or the neural-networks-based method~\cite{takeishi2017learning}.

In this work, we propose a new algorithm for construction of the Koopman eigenfunctions from data. The method is geared toward transient, off-attractor, dynamics where the spectrum of the Koopman operator is \emph{extremely rich}. In particular, provided that a \emph{non-recurrent} surface exists in the state-space,  \emph{any} complex number is an eigenvalue of the Koopman operator with an associated continuous (or even smooth if so desired) eigenfunction, defined everywhere on the image of this non-recurrent surface through the flow of the dynamical system. What is more, the associated eigenspace is infinite-dimensional, parametrized by functions defined on the boundary of the non-recurrent surface. We leverage this richness to obtain a large number of eigenfunctions in order to ensure that the observable quantity of interest (e.g., the state itself) lies \emph{within the span of the eigenfunctions} (and hence within an invariant subspace of the Koopman operator) and is therefore predictable in a \emph{linear} fashion. The requirement that the embedding mapping spans an invariant subspace \emph{and} the quantity of interest belongs to this subspace are crucial for practical applications: they imply both a \emph{linear time evolution} in the embedding space as well the possibility to \emph{reconstruct} the quantity of interest in a \emph{linear fashion}. On the other hand, using a nonlinear reconstruction mapping may lead to comparatively low-dimensional embeddings but it does not offer a significant advantage when used within an optimization-based control or estimation scheme since in this case one non-convex problem is replaced by another.

In addition to eigenfunctions, the proposed method can be extended to construct \emph{generalized eigenfunctions} that also give rise to Koopman invariant subspaces and can hence be used for linear prediction; this further enriches the class of embedding mappings constructible using the proposed method.
 
On an algorithmic level, given a set of initial conditions lying on distinct trajectories, a set of complex numbers (the eigenvalues) and a set of continuous functions, the proposed method constructs eigenfunctions by simply ``flowing'' the values of the continuous functions forward in time according the eigenfunction equation, starting from the values of the continuous functions defined on the set of initial conditions. Provided the trajectories are non-periodic, this consistently and uniquely defines the eigenfunctions on the entire data set. These eigenfunctions are then extended to the entire state-space by interpolation or approximation. We prove that such extension is possible (i.e., there exist continuous eigenfunctions taking the computed values on the data set) provided that there is a non-recurrent surface passing through the initial conditions of the trajectories and we prove that such surface always exists provided the flow is rectifiable in the considered time interval. We also prove that the eigenfunctions constructed in this way span the space of all continuous functions in the limit as the number of boundary function-eigenvalue pairs tends to infinity. This implies that in the limit any continuous observable can be arbitrarily accurately approximated by \emph{linear} combinations of the eigenfunctions, a crucial requirement for practical applications.

\new{Importantly, both the values of the boundary functions and the eigenvalues can be selected using numerical optimization.  The minimized objective is simply the projection error of the observables of interest onto the span of the eigenfunctions. For the problem of boundary function selection, we derive a convex reformulation, leading to a linear least-squares problem. The problem of eigenvalue selection appears to be intrinsically non-convex but is fortunately relatively low-dimensional, thereby amenable to a wide array of local and global non-convex optimization techniques. This is all that is required to construct the linear predictors in the uncontrolled setting.}

In the controlled setting, we follow a two-step procedure. First, we construct a predictor for the uncontrolled part of the system (i.e., with the control being zero or any other fixed value). Next, using a second data set generated with control we minimize a \emph{multi-step} prediction error in order to obtain the input matrix for the linear predictor. Crucially, the multi-step error minimization boils down to a simple \emph{linear} least-squares problem; this is due to the fact that the dynamics and output matrices are already identified. This is a distinctive feature of the approach, compared to (E)DMD-based methods (e.g., \cite{korda2018convergence,brunton_dmd_control}) where only a one-step prediction error can be minimized in a convex fashion. 

The predictors obtained in this way are then applied within the Koopman model predictive control (Koopman MPC) framework of~\cite{korda2018linear}, which we briefly review in this work. \new{A core component of any MPC controller is the minimization of an objective functional over a multi-step prediction horizon; this is the primary reason for using a multi-step prediction error minimization during the predictor construction.}  However, the eigenfunction and linear predictor construction methods are completely general and immediately applicable, for example, in the state estimation setting~\cite{surana_estim,surana2017koopman}.

\new{The predictors obtained can then applied within the Koopman model predictive control (Koopman MPC) framework (see~\cite{korda2018linear} for a general theory and~\cite{arbabi2018data, korda2018power} for applications in fluid mechanics and power grid control). A core component of any MPC controller is the minimization of an objective functional over a multi-step prediction horizon; this is the primary reason for using a multi-step prediction error minimization during the predictor construction.}  However, the eigenfunction and linear predictor construction methods are completely general and immediately applicable, for example, in the state estimation setting~\cite{surana_estim,surana2016koopman}.

The fact that the spectrum of the Koopman operator is very rich in the space of continuous functions is a well known fact in the Koopman operator community; see, e.g., \cite[Theorem 3.0.2]{kuster2015koopman}. In particular, the fact that, away from singularities, eigenfunctions corresponding to \emph{arbitrary} eigenvalues can be constructed was noticed in~\cite{mezic2017koopman} where these were termed \emph{open eigenfunctions} and they were subsequently used in~\cite{bollt2018matching} to find conjugacies between dynamical systems. This work is, to the best of our knowledge, the first one to exploit the richness of the spectrum for prediction and control using \emph{linear} predictors and to provide a theoretical analysis of the set of eigenfunctions obtained in this way. On the other hand, the spectrum of the Koopman operator ``on attractor'', in a post-transient regime, is much more structured and can be analyzed numerically in a great level of detail (see, e.g., \cite{korda2018data,govindarajan2018approximation}).

\paragraph{Notation} The set of real numbers is denoted by $\Rb$, the set of complex numbers by $\Cb$, the set of natural numbers by $\Nb = \{1,2,\ldots\}$ and $\Nb_0 = \Nb \cup \{0\}$. The space of continuous \new{complex-valued} functions defined on a set $X\subset \Rb^n$ is denoted by \new{$\C(X)$ or $\C(X;\Cb)$, whenever we want to emphasize that the codomain is complex}. The symbol~$\circ$ denotes the pointwise composition of two functions, i.e., $(g\circ f)(x) = g(f(x))$. The Moore-Penrose pseudoinverse of a matrix ${\bf A} \in \Cb^{n\times n}$ is denoted by ${\bf A}^\dagger$, the transpose by ${\bf A}^\top$  and the conjugate (Hermitian) transpose by ${\bf A}\herm$. The identity matrix is denoted by ${\bf I}$. The symbol $\mr{diag}(\cdot,\ldots,\cdot)$ denotes a (block-) diagonal matrix composed of the arguments. The symbol $\|x\|_2$ denotes the Euclidean norm and $\|x\|_\infty$ the max norm of a vector $x\in \Rb^n$.

\section{Koopman operator}
We first develop our framework for uncontrolled dynamical systems and generalize it to controlled systems in Section~\ref{sec:cont}. Consider therefore the nonlinear dynamical system
\begin{equation}\label{eq:sys}
	\dot{x} =  f(x)
\end{equation}
with the state $x\in X \subset \Rb^n$ and $f$ Lipschitz continuous on $X$. The flow of this dynamical system is denoted by $S_t(x) $, i.e.,
\begin{equation}\label{eq:flow}
	\frac{d}{dt}S_t(x) = f(S_t(x)),
\end{equation}
\new{which we assume to be well defined} for all $x \in X$ and all $t\ge 0$. The \emph{Koopman operator semigroup} $(\Kc_t)_{t \ge 0}$ is defined by
\[
\Kc_t g = g \circ  S_t
\]
for all $g \in \C(X)$. Since the flow of a dynamical system with Lipschitz vector field is also Lipschitz, it follows that
\[
\Kc_t : \C(X) \to \C(X),
\]
i.e., each element of the Koopman semigroup maps continuous functions to continuous functions. Crucially for us, each $\Kc_t$ is a \emph{linear} operator\footnote{\new{To see the linearity of $\Kc_t$ consider $g_1\in \C(X)$, $g_2\in \C(X)$ and $\alpha \in \Cb$. Then we have $\Kc_t(\alpha g_1 + g_2) = (\alpha g_1 + g_2)\circ S_t = \alpha g_1\circ S_t + g_2\circ S_t = \alpha \Kc_t g_1 + \Kc_tg_2$.}}.

With a slight abuse of language, from here on, we will refer to the Koopman operator semigroup simply as the \emph{Koopman operator}.

\paragraph{Eigenfunctions} An \emph{eigenfunction} of the Koopman operator associated to an eigenvalue $\lambda \in \Cb$ is any nonzero function $\phi \in \new{\C(X;\Cb)}$ satisfying
\begin{equation}\label{eq:efun_def1}
(\Kc_t \phi)(x) = e^{\lambda t}\phi(x),
\end{equation}
which is equivalent to
\begin{equation}\label{eq:efun_def}
\phi(S_t(x)) = e^{\lambda t}\phi(x).
\end{equation}
Therefore, any such eigenfunction defines a coordinate evolving \emph{linearly} along the flow of~(\ref{eq:sys}) and satisfying the \emph{linear} ordinary differential equation (ODE)
\begin{equation}\label{eq:KoopmanEigfun_ODE}
\frac{d}{dt}\phi(S_t(x)) = \lambda \phi(S_t(x)).
\end{equation}

\subsection{Linear predictors from eigenfunctions}
Since the eigenfunctions define linear coordinates, they can be readily used to construct \emph{linear} predictors for the nonlinear dynamical system~(\ref{eq:sys}).  The goal is to predict the evolution of a quantity of interest $\ob (x)$ (often referred to as an ``observable'' or an ``output'' of the system) along the trajectories of~(\ref{eq:sys}). The function 
\[
\ob : \Rb^n\to \Rb^{n_\ob}
\]
often represents the state itself, i.e., $\ob(x) = x$ or an output of the system (e.g., the attitude of a vehicle or the kinetic energy of a fluid) or the cost function to be minimized within an optimal control problem or a nonlinear constraint on the state of the system (see Section~\ref{sec:KoopmanMPC} for concrete examples). The distinctive feature of this work is the requirement that the predictor constructed be a \emph{linear} dynamical system. This facilitates the use of \emph{linear} tools for state estimation and control, thereby greatly simplifying the design procedure as well as drastically reducing computational and deployment costs (see \cite{surana_estim} for applications of this idea to state estimation and \cite{korda2018linear} for model predictive control).

Let $\phi_1,\ldots,\phi_N$ be eigenfunctions of the Koopman operator with the associated (not necessarily distinct) eigenvalues $\lambda_1,\ldots,\lambda_N$. Then we can construct a linear predictor of the form
\begin{subequations}\label{eq:linPred_uncont}
\begin{align}
\dot{z} &=  Az \\
z_0 &= \bs \phi(x_0), \\
\hat{y} &=  Cz,
\end{align}
\end{subequations}
where
\begin{equation}\label{eq:A_and_phi}
A = \begin{bmatrix}
    \lambda_{1} & & \\
    & \ddots & \\
    & & \lambda_{N}
  \end{bmatrix}, \qquad \bs \phi = \begin{bmatrix}
\phi_1\\\vdots \\ \phi_N
\end{bmatrix},
\end{equation}
and where $\hat y$ is the prediction of $\ob(x)$. To be more precise, the prediction of $\ob(x(t)) = \ob(S_t(x_0))$ is given by
\[
\ob(x(t)) \approx \hat y (t) = C 
e^{A t} z_0 = C\begin{bmatrix}
   e^{\lambda_{1}t} & & \\
    & \ddots & \\
    & & e^{\lambda_{N}t}
    \end{bmatrix}z_0.
\]
The matrix $C$ is chosen such that the projection of $\ob$ onto $\mr{span}\{\phi_1,\ldots,\phi_N\}$ is minimized, i.e., $C$ solves the optimization problem
\begin{equation}\label{opt:proj}
\min_{C\in \Cb^{n_\ob\times N}} \|\ob - C\bs\phi \|,
\end{equation}
where $\| \cdot\|$ is a norm on the space of continuous functions (e.g., the sup-norm or the $L_2$ norm).

\paragraph{Prediction error} Since $\phi_1,\ldots,\phi_N$ are the eigenfunctions of the Koopman operator, the prediction of the evolution of the eigenfunctions along the trajectory of~(\ref{eq:sys}) is error-free, i.e.,
\[
z(t) = \bs\phi(x(t)).
\]
Therefore, the \emph{sole source} of the prediction error
\begin{equation}\label{eq:predError}
\|\ob (x(t)) - \hat y(t) \|
\end{equation}
is the error in the projection of $\ob$ onto $\mr{span}\{\phi_1,\ldots,\phi_N\}$, quantified by Eq.~(\ref{opt:proj}). In particular, if $\ob\in \mr{span}\{\phi_1,\ldots,\phi_N\}$ (with the inclusion understood componentwise), we have
\[
\|\ob (x(t)) - \hat y(t) \| = 0\quad \forall\, t\ge 0.
\]
This observation\footnote{\new{To see this, notice that $\ob$ belonging to the span of $\bs\phi$ is equivalent to the existence of a matrix $C$ such that $\ob = C\bs\phi$. Therefore, $\ob(x(t)) = \ob\circ S_t = C\bs\phi\circ S_t = Ce^{At}\bs\phi = \hat y(t)$.}} will be crucial for defining a meaningful objective function when learning the eigenfunctions from data.

\new{\paragraph{Goals} The primary goal of this paper is a data-driven construction of a set of eigenfunctions $\{\phi_1,\ldots,\phi_N\}$ such that the error~(\ref{opt:proj}) is minimized. In doing so, we introduce and theoretically analyze a novel eigenfunction construction procedure (Section~\ref{sec:nonRecur}) from which a data-driven optimization-based algorithm is derived in Section~\ref{sec:learning}. A secondary goal of this paper is to use the constructed eigenfunctions for prediction in a controlled setting (Section~\ref{sec:cont}) and subsequently for model predictive control design (Section~\ref{sec:KoopmanMPC}) within the Koopman MPC framework of~\cite{korda2018linear}.}

\section{Non-recurrent sets and eigenfunctions}\label{sec:nonRecur}
In this section we show how non-recurrent sets naturally give rise to eigenfunctions. Let time \new{$T \in (0,\infty)$} be given. A set $\Gamma \subset X$ is called non-recurrent if
\[
x \in \Gamma \Longrightarrow S_t(x) \notin \Gamma\quad \forall t\in (0,T].
\]
Given \emph{any} function $g\in \C(\Gamma)$ and \emph{any} $\lambda \in \Cb$, we can \emph{construct} an eigenfunction of the Koopman operator by simply solving the defining ODE~(\ref{eq:KoopmanEigfun_ODE}) ``initial condition by initial condition'' for all initial conditions $x_0 \in \Gamma$. Mathematically, we define for all $x_0 \in \Gamma$
\begin{equation}\label{eq:phi_comp_def}
\phi_{\lambda,g}(S_t(x_0)) = e^{\lambda t} g(x_0),
\end{equation}
which defines the eigenfunction on the entire image $X_T$ of $\Gamma$ by the flow $S_t(\cdot)$ for $t\in [0,T]$. Written explicitly, this image is
\[
X_T = \bigcup_{t\in [0,T]} S_t(\Gamma) = \bigcup_{t\in [0,T]} \{ S_t(x_0) \mid x_0\in \Gamma\}.
\]
See Figure~\ref{fig:nonRecur_basic} for an illustration. 
\begin{figure*}[tb]
	\begin{picture}(50,140)
		\put(120,0){\includegraphics[width=100mm]{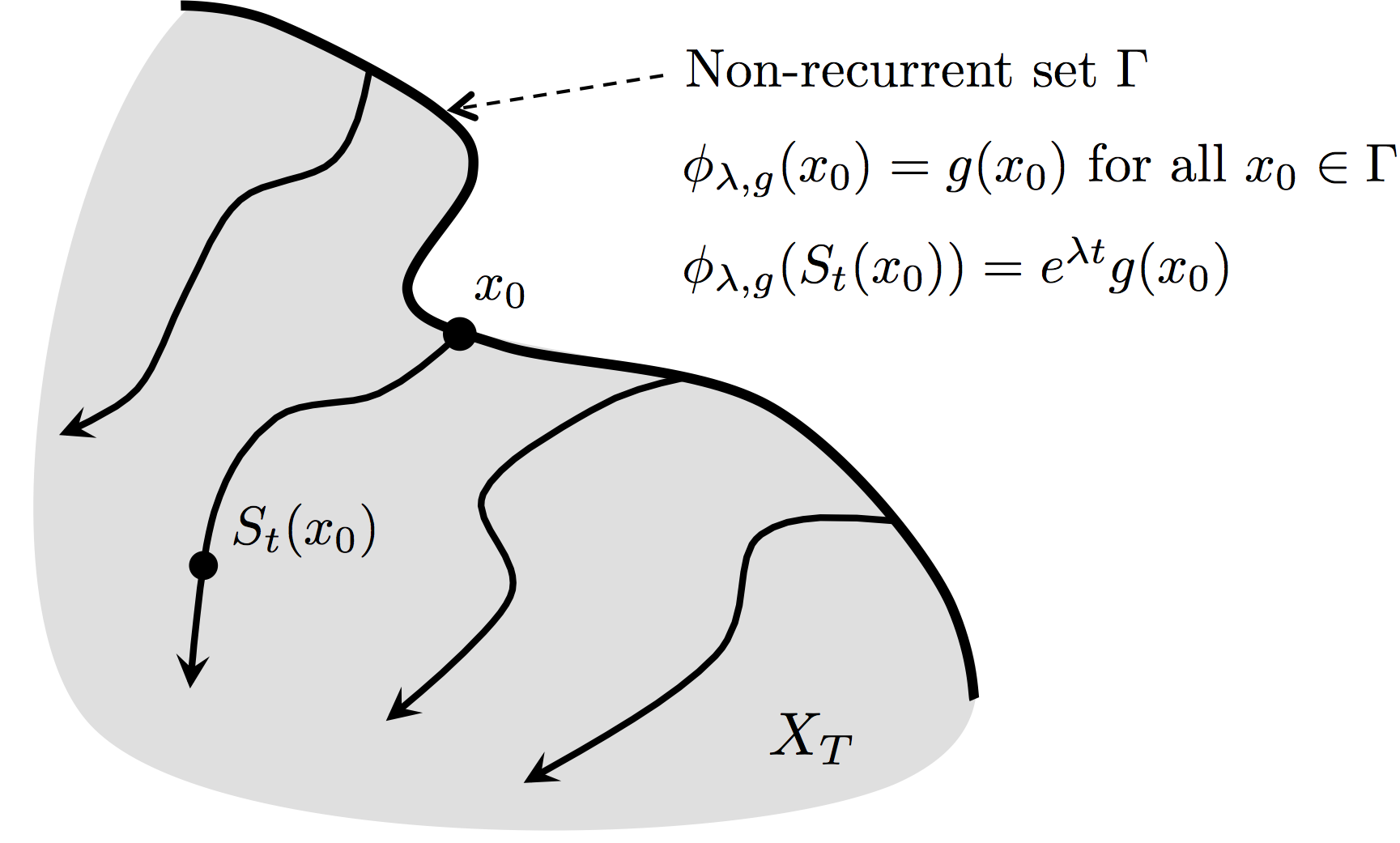}}
	\end{picture}
	\caption{\footnotesize Construction of eigenfunctions using a non-recurrent set $\Gamma$ and a continuous function $g$ defined on $\Gamma$.}
	\label{fig:nonRecur_basic}
\end{figure*}
To get an explicit expression for $\phi_{\lambda,g}(x)$ we flow backward in time until we hit the non-recurrent set $\Gamma$, obtaining
\begin{equation}\label{eq:phi_def}
\phi_{\lambda,g}(x) = e^{-\lambda\tau(x)}g(S_{\tau(x)}(x))
\end{equation}
for all $x\in X_T$, where
\[
\tau(x) = \inf_{t \in \Rb}\{t\mid S_{t}(x) \in \Gamma\}
\]
is the first time that the trajectory of~(\ref{eq:sys}) hits $\Gamma$ starting from $x$. By construction, for $x\in X_T$ we have 
\[
\tau(x) \in [-T,0].
\]
The results of this discussion are summarized in the following theorem:
\begin{theorem}
Let $\Gamma$ be a non-recurrent set, $g\in \C(\Gamma)$ and $\lambda \in \Cb$. Then $\phi_{\lambda,g}$ defined by~(\ref{eq:phi_def}) is an eigenfunction of the Koopman operator on $X_T$. In particular, $\phi_{\lambda,g}$ satisfies (\ref{eq:efun_def}) and (\ref{eq:KoopmanEigfun_ODE}) for all $x\in X_T$ and all $t$ such that $S_t(x) \in X_T$. In addition, if $g$ is Lipschitz continuous, then also
\begin{equation}\label{eq:KoopEfun_pde}
\nabla \phi_{\lambda,g}\cdot f = \lambda \phi_{\lambda,g}
\end{equation}
almost everywhere in $X_T$ (and everywhere in $X_T$ if $g$ and $f$ are differentiable).
\end{theorem}
\begin{proof}
The result follows by construction. Since $\Gamma$ is non-recurrent, the definition~(\ref{eq:phi_comp_def}) is consistent for all $t\in [0,T]$ and equivalent to~(\ref{eq:phi_def}). Since $\new{S_0(x_0) = x_0}$, we have $\phi_{\lambda,g}(x_0) = g(x_0)$ for all $x_0\in \Gamma$ and hence Eq.~(\ref{eq:phi_comp_def}) is equivalent to Eq.~(\ref{eq:efun_def}) defining the Koopman eigenfunctions. To prove~(\ref{eq:KoopEfun_pde}) observe that $g$ Lipschitz implies that $\phi_{\lambda,g}$ is Lipschitz and the result follows from~(\ref{eq:KoopmanEigfun_ODE}) by the chain rule and the Rademacher theorem, which asserts almost-everywhere differentiability of Lipschitz functions.
\end{proof}

Several remarks are in order.

\paragraph{Richness} We emphasize that this construction works for an \emph{arbitrary} $\lambda \in \Cb$ and an \emph{arbitrary} function $g$ continuous\footnote{In this work we restrict our attention to functions $g$ continuous on $\Gamma$ but in principle discontinuous functions of a suitable regularity class could be used as well.} on $\Gamma$. Therefore, there are uncountably many eigenfunctions that can be generated in this way and in this work we exploit this to construct a sufficiently rich collection of eigenfunctions such that the projection error~(\ref{opt:proj}) is minimized. The richness of the class of eigenfunctions is analyzed theoretically in Section~\ref{sec:density} and used practically in Section~\ref{sec:learning} for data-driven learning of eigenfunctions.

\paragraph{Time direction} The same construction can be carried out backwards in time or forward and backward in time, as long as $\Gamma$ is non-recurrent for the time interval considered. In this work we focus on forward-in-time construction which naturally lends itself to data-driven applications where typically only forward-in-time data is available.

\paragraph{History} This construction is very closely related to the concept of \emph{open eigenfunctions} introduced in~\cite{mezic2017koopman}, which were subsequently used in~\cite{bollt2018matching} to find conjugacies between dynamical systems. This work is, to the best of our knowledge, the first one to use such construction for prediction and control using linear predictors.

\subsection{Non-recurrent set vs Non-recurrent surface} It is useful to think of the non-recurrent set $\Gamma$ as an $n-1$ dimensional surface so that $X_T$ is full dimensional. Such surface can be for example any level set of a Koopman eigenfunction with non-zero real part (e.g., isostable) or a level set of a Lyapunov function. However, these level sets can be hard to obtain in practice; fortunately, their knowledge is \emph{not required}. The reason for this is that the set~$\Gamma$ can be a finite \emph{discrete} set in which case $X_T$ is simply the collection of all trajectories with initial conditions in $\Gamma$; since trajectories are one-dimensional, any randomly generated finite (or countable) discrete set will be non-recurrent on \new{$[0,T]$} with probability one (unless the dynamics is periodic). This is a key feature of our construction that will be utilized in Section~\ref{sec:learning} for a data-driven learning of the eigenfunctions. A natural question arises: can one find a non-recurrent surface passing through a given finite discrete non-recurrent set $\Gamma$? The answer is positive, provided that the points in $\Gamma$ do not lie on the same trajectory and the flow can be rectified:
\begin{lemma}\label{lem:nonRecurSurface}
Let $\Gamma = \{x^1,\ldots, x^M\}$ be a finite set of points in $X$ and let $X'$ be a full dimensional compact set containing $\Gamma$ on which the flow of~(\ref{eq:sys}) can be rectified, i.e., there exists a diffeomorphism $\conjug: Y' \to X'$ through which~(\ref{eq:sys}) is conjugate\footnote{\new{Two dynamical systems $\dot{x} = f_1(x)$ and $\dot{y} = f_2(y)$ are conjugate through a diffeomorphism $\zeta$ if the associated flows $S_{1,t}$ and $S_{2,t}$ satisfy $S_{1,t}(x)= \zeta(S_{2,t}(\zeta^{-1}(x)))$. For the vector fields, this means that $f_2(y) = (\frac{\partial \zeta}{\partial y})^{-1}f_1(\zeta(y))$.}} to 
\begin{equation}\label{eq:rectified1}
\dot{y} = (0,\ldots,0,1)
\end{equation}
with $Y' \subset \Rb^n$ convex. Assume that no two points in the set $\Gamma$ lie on the same trajectory of~(\ref{eq:sys}). Then there exists an $n-1$ dimensional surface $\hat \Gamma \supset \Gamma$, closed in the standard topology of $\Rb^n$, such that $x\in \hat\Gamma$ implies $S_t(x) \notin \hat\Gamma$ for any $t > 0$ satisfying $S_{t'}(x) \in X'$ for all $t'\in [0,t]$.
%and that the trajectories starting from $\Gamma$ are not periodic (i.e., $x\in \Gamma$ implies $S_t(x)\notin \Gamma$ for all $t \ne 0$).
\end{lemma}
\begin{proof}
Let $y^j = \conjug^{-1}(x^j)$, $j=1,\ldots, M$ and let $\Gamma_Y = \{y^1,\ldots, y^M\} = \conjug^{-1}(\Gamma)$. The goal is to construct an $n-1$ dimensional surface $\hat\Gamma_Y$, closed in $\Rb^n$, passing through the points $\{y^1,\ldots, y^M\}$, i.e., $\hat\Gamma_Y \supset \Gamma_Y$ and satisfying 
\begin{equation}\label{eq:nonRec_proof_aux}
y\in \hat\Gamma_Y \Longrightarrow \hat S_t(y) \notin \hat\Gamma_Y
\end{equation}
for any $t > 0$ satisfying $\hat S_{t'}(y) \in Y'$ for all $t'\in[0,t]$, where $\hat S_{t'}(y)$ denotes the flow of~(\ref{eq:rectified1}). Once $\hat\Gamma_Y$ is constructed, the required surface $\hat\Gamma$ is obtained as $\hat\Gamma = \conjug(\hat\Gamma_Y)$.

Given the nature of the rectified dynamics~(\ref{eq:rectified1}), the condition~(\ref{eq:nonRec_proof_aux}) will be satisfied if $\hat\Gamma_Y$ is a graph of a Lipschitz continuous function $\gamma:\Rb^{n-1}\to \Rb$ such that
\[
\hat\Gamma_Y = \{(y_1,\ldots,y_n) \in Y' \mid y_n = \gamma(y_1,\ldots,y_{n-1})\}.
\]

We shall construct such function $\gamma$. Denote $\bar{y}^j = (y^j_1,\ldots, y^j_{n-1})$ the first $n-1$ components of each point $y^j \in \Rb^n$. The nature of the rectified dynamics~(\ref{eq:rectified1}), convexity of $Y'$ and the fact that $x^j$'s (and hence $y^j$'s) do not lie on the same trajectory implies that $\bar{y}^j$'s are distinct. Therefore, the pairs $(\bar y^j, y^j_n)\in \Rb^{n-1}\times\Rb$, $j=1,\ldots, M$, can be interpolated with a Lipschitz continuous function $\gamma:\Rb^{n-1}\to \Rb$. One such example of $\gamma$ is 
\[
\gamma(y_1,\ldots,y_{n-1}) = \max_{j \in \{1,\ldots,M\}}\Big\{   \big[1- \|(y_1,\ldots,y_{n-1}) -\bar y^j\|\big] y^j_n  \Big\},
\]
where we assume that $y^j_n \ge 0$ (which can be achieved without loss of generality by translating the $y_n$-th coordinate since $Y'$ is compact) and $\|\cdot\|$ is any norm on $\Rb^{n-1}$. Another example is a multivariate polynomial interpolant of degree $d$ which always exists for any $d$ satisfying $\binom{n-1+d}{d} \ge M$. Since both $\conjug$ and $\gamma$ are Lipschitz continuous, the surface $\hat \Gamma$ is $n-1$ dimensional and is closed in the standard topology of $\Rb^n$.
\end{proof}

\begin{remark}[Rectifyability] It is a well-known fact that for the  dynamics~(\ref{eq:sys}) to be rectifyable on a  domain $X'$, \new{the vector field $f$ should be non-singular on this domain (i.e., $f(x) \ne 0$ for all $x\in X'$)}; see, e.g., \cite[Chapter 2, Corollary~12]{arnoldODEs}.

% Note, however, that the assumption of rectifyability in Theorem~\ref{thm:mainDensityThm} is sufficient but not necessary for the conclusions of the theorem to hold. An analogous result can be proven, for example, in the basin of attraction of a stable equilibrium, where the flow is not rectifyable and we conjecture that similar result holds in a neighborhood of more general singularities. The only fundamental obstruction is recurrence, which restricts the set of  eigenvalues of the Koopman operator and hence also restricts the richness of the class of its eigenfunctions.
\end{remark}

\subsection{Span of the eigenfunctions}\label{sec:density}
A crucial question arises: can one approximate an arbitrary continuous function by a linear combination of the eigenfunctions constructed using the approach described in Section~\ref{sec:nonRecur} by selecting more and more boundary functions $g$ and eigenvalues $\lambda$? Crucially for our application, if this is the case, we can make the projection error~(\ref{opt:proj}) and thereby also the prediction error~(\ref{eq:predError}) arbitrarily small by enlarging the set of eigenfunctions $\bs \phi$. If this is the case, does one have to enlarge the set of eigenvalues or does it suffice to only increase the number of boundary functions $g$? In this section we give a precise answer to these questions.

Before we do so, we set up some notation. Given any set $\Lambda \subset \Cb$, we define
\begin{equation}\label{eq:mesh}
\lattice(\Lambda) = \Big\{ \sum_{k=1}^{p} \alpha_k \lambda_k \mid \lambda_k \in \Lambda,\; \alpha_k\in \Nb_0, \; p\in \Nb \Big\}.
\end{equation}
A basic result in the Koopman operator theory asserts that if $\Lambda$ is a set of eigenvalues of the Koopman operator, then so is $\lattice(\Lambda) $. Now, given $\Lambda \subset \Cb$ and $G\subset \C(\Gamma)$, we define
\begin{equation}\label{eq:PhiLamG}
\Phi_{\Lambda,G} =  \{ \phi_{\lambda,g} \mid \lambda  \in \Lambda, \; g\in G  \},
\end{equation}
where $\Phi_{\lambda,g}$ is given by~(\ref{eq:phi_def}). In words, $\phi_{\Lambda,G} $ is the set of all eigenfunctions arising from all combinations of boundary functions in $G$ and eigenvalues in $\Lambda$ using the procedure described in Section~\ref{sec:nonRecur}.

Now we are ready to state the main result of this section:

%\begin{theorem}\label{thm:mainDensityThm}
%Let $\Gamma$ be a non-recurrent set, closed in the standard topology of $\Rb^n$, and let the vector field $f$ be rectifiable in $X_T$, i.e., the dynamics (\ref{eq:sys}) is conjugate to
%\begin{equation}\label{eq:rectified}
%\dot{y} = (0,\ldots,0,1)
%\end{equation}
%through a homeomorphism $h: Y_T \to X_T$. Let $\Lambda_0 \subset \mathbb{C}$ be an arbitrary\footnote{In the extreme case, the set $\Lambda_0$ can consists of a single non-zero real number or a single conjugate pair with non-zero real part.} set of complex numbers such that at least one has a non-zero real part and $\Lambda_0 = \bar\Lambda_0$. Set $\Lambda = \lattice(\Lambda_0)$ and let $G = \{g_i\}_{i=1}^\infty$ denote an arbitrary set of functions whose span is dense in $\C(\Gamma)$. Then the span of $ \Phi_{\Lambda,G} $ is dense in $\C(X_T)$, i.e., for every $\xi \in \C(X_T)$ and any $\epsilon > 0$ there exist eigenfunctions $\phi_1,\ldots, \phi_N \in \Phi_{\Lambda,G} $ such that
%\[
%\sup_{x\in X_T} \left| \xi(x) - \sum_{i=1}^N c_i \phi_i(x)   \right| < \epsilon.
%\]
%for some $c_i \in \Cb$.
%\end{theorem}

\begin{theorem}\label{thm:mainDensityThm}
Let $\Gamma$ be a non-recurrent set, closed in the standard topology of $\Rb^n$ and let $\Lambda_0 \subset \mathbb{C}$ be an arbitrary\footnote{In the extreme case, the set $\Lambda_0$ can consists of a single non-zero real number or a single conjugate pair with non-zero real part.} set of complex numbers such that at least one has a non-zero real part and $\Lambda_0 = \bar\Lambda_0$. Set $\Lambda = \lattice(\Lambda_0)$ and let $G = \{g_i\}_{i=1}^\infty$ denote an arbitrary set of functions whose span is dense in $\C(\Gamma)$ in the supremum norm. Then the span of $ \Phi_{\Lambda,G} $ is dense in $\C(X_T)$, i.e., for every $h \in \C(X_T)$ and any $\epsilon > 0$ there exist eigenfunctions $\phi_1,\ldots, \phi_N \in \Phi_{\Lambda,G} $ and complex numbers $c_1,\ldots,c_N$ such that
\[
\sup_{x\in X_T} \left| h(x) - \sum_{i=1}^N c_i \phi_i(x)   \right| < \epsilon.
\]
\end{theorem}
\new{
\begin{proof}
 \textbf{Step 1}. First, we observe that it is sufficient to prove the density of
\[
\Phi_{\Lambda} = \mr{span}\{ \phi_{\lambda,g} \mid \lambda  \in \Lambda, \; g\in \C(\Gamma)  \}
\]
in $\C(X_T)$. To see this, assume $\Phi_{\Lambda} $ is dense in $\C(X_T)$  and consider any function $h \in \C(X_T)$ and $\epsilon > 0$. Then there exist eigenfunctions $\phi_{\lambda_i,g_i}$, $i=1,\ldots,k$, defined by~(\ref{eq:phi_def}) with $g_i\in \C(\Gamma)$ and $\lambda_i \in \Lambda$ such that
\[
 \sup_{x\in X_T} \Big | \sum_{i=1}^k \phi_{\lambda_i,g_i}(x) - h(x) \Big | < \epsilon.
\] 
Here we used the linearity of (\ref{eq:phi_def}) with respect to $g$ to subsume the coefficients of the linear combination to the functions $g_i$.

Since $\mr{span}\{G\}$ is dense in $\C(\Gamma)$, there exist functions $\tilde g_i \in \mr{span}\{G\}$ such that \[
\sup_{x\in \Gamma} |g_i - \tilde g_i | < \frac{\epsilon}{k}\min\{1,|e^{\lambda_i T}|\} .\]
In addition, because Eq. (\ref{eq:phi_def}) defining $\phi_{\lambda,g}$~is linear in $g$ for any fixed $\lambda$, it follows that $\phi_{\lambda_i,\tilde g_i} \in \mr{span}\{\Phi_{\Lambda,G}\}$ and hence also $\sum_{i=1}^k \phi_{\lambda_i,\tilde g_i} \in \mr{span}\{\Phi_{\Lambda,G}\}$. Therefore it suffices to bound the error between $h$ and $\sum_{i=1}^k\phi_{\lambda_i,\tilde g_i}$. We have
\begin{align*}
\sup_{x\in X_T} \left|\sum_{i=1}^k \phi_{\lambda_i,\tilde g_i}(x) - h(x) \right|  & \le  \sup_{x\in X_T}\left|\sum_{i=1}^k\phi_{\lambda_i, g_i}(x) - h(x) \right|  + \sup_{x\in X_T} \left| \sum_{i=1}^k (\phi_{\lambda_i, g_i}(x) - \phi_{\lambda_i,\tilde g_i}(x)) \right| \\ & \le \epsilon + \sup_{x\in X_T} \left|\sum_{i=1}^k e^{-\lambda_i\tau(x)}[g(S_{\tau(x)}(x)) - \tilde g_i(S_{\tau(x)}(x))] \right | \\ & \le 
\epsilon + \sum_{i=1}^k \left[    \sup_{x\in X_T}|e^{-\lambda_i\tau(x)}| \cdot \sup_{x\in \Gamma} |g_i(x) - \tilde g_i(x)| \right] \\& \le \epsilon + \frac{\epsilon}{k} \sum_{i=1}^k \max\{1,|e^{-\lambda_i T}|\} \min\{1,|e^{\lambda_i T}|\}  \le 2\epsilon,
\end{align*}
where we used the facts that $\tau(x)\in[0,T]$ and $S_{\tau(x)}(x) \in \Gamma$.

\textbf{Step 2}.  We will show that $\Phi_{\Lambda}$ is a subalgebra of $\C(X_T)$ closed under complex conjugation that separates points and contains a non-zero constant function; then the Stone–Weierstrass theorem will imply the desired results. By construction, $\Phi_\Lambda$ is a linear subspace of $\C(X_T)$ and hence it suffices to show that $\Phi_\Lambda$ is closed under multiplication and complex conjugation, separates points and contains a non-zero constant function.

To see that that $\Phi_\Lambda$ is closed under multiplication, consider $\phi_1 \in \Phi_\Lambda$ and $\phi_2 \in \Phi_\Lambda$ of the form $\phi_1 = \sum_{i} \phi_{\lambda_i,g_i}$ and $\phi_1 = \sum_{i}  \phi_{\lambda_i',g_i'}$ with $\lambda_i\in \Lambda$ and $\lambda_i'\in \Lambda$, $g_i \in \C(\Gamma)$, $g_i' \in \C(\Gamma)$. Then we have
\[
\phi_1\phi_2 = \sum_{i,j}   \phi_{\lambda_i,g_i} \phi_{\lambda_j',g_j'} = \sum_{i,j} e^{-\lambda_i \tau}e^{-\lambda_j' \tau}   (g_i \circ S_{\tau})  (g_j' \circ S_{\tau}) = \sum_{i,j} \phi_{\lambda_i+\lambda_j',g_ig_j'} \in \Phi_\Lambda  
\]
since $\lambda_i + \lambda_j' \in \Lambda$ because of~(\ref{eq:mesh}) and $g_i'g_j' \in \C(\Gamma)$.

To see that $\Phi_\Lambda$ separates points of $X_T$, i.e., for each $x_1 \in X_T$ and $x_2 \in X_T$, $x_1\ne x_2$, there exists $\phi \in \Phi_\Lambda$ such that $\phi(x_1) \ne \phi(x_2)$. We consider two cases. First, suppose that $x_1$ and $x_2$ lie on the same trajectory of~(\ref{eq:sys}). Then these two points are separated by $\phi_{\lambda,1}$ for any $\lambda$ with nonzero real part; by the assumptions of the theorem there exists $\lambda \in \Lambda$ with a non-zero real part and hence the associated $\phi_{\lambda,1}$ belongs to $\Phi_\Lambda$. Second, suppose $x_1$ and $x_2$ do not lie on the same trajectory. Then these two points are separated by $\phi_{0,g}$ with $g(S_{\tau(x_1)}(x_1)) \ne g(S_{\tau(x_2)}(x_2))$; such $g$ always exists in $\C(\Gamma)$ because $\C(\Gamma)$ separates points of~$\Gamma$.

To see that $\Phi_\Gamma$ contains a constant non-zero function, consider $\phi_{\lambda,g} $ with $\lambda = 0$ and $g = 1$, which is equal to $1$ on $X_T$ and belongs to $\Phi_\Lambda$.

Finally, to see that $\Phi_\Lambda$ is closed under complex conjugation, consider $\phi = \sum_i \phi_{\lambda_i,g_i}$, $\lambda_i\in \Lambda$ and $g_i \in \C(\Gamma)$. Then
\[
\bar\phi = \sum_i \overline{\phi_{\lambda_i,g_i}}  = \sum_i e^{-\bar \lambda_i \tau }\overline{g_i} \circ S_{\tau} = \sum_i \phi_{\bar\lambda_i, \overline{g_i}} \in \Phi_\Lambda
\]
since $\bar\Lambda = \Lambda$ by assumption and $\C(\Gamma) = \overline{\C(\Gamma)}$.
\end{proof}
}

\paragraph{Selection of $\lambda$'s and $g$'s}
\new{An interesting question arises regarding an optimal selection of the eigenvalues $\lambda\in \Lambda$ and boundary functions $g \in G$ assuring that the projection error~(\ref{opt:proj}) converges to zero as fast as possible. As it turns, the boundary functions $g$ can be chosen optimally using convex optimization, for each component of $\ob$ separately. The optimal choice of the eigenvalues $\lambda$ appears to be more difficult and seems to be inherently non-conconvex. Choices of both $g$'s and $\lambda$'s using optimization are discussed in detail in Sections~\ref{sec:optimChoice_g} and~\ref{sec:optimChoice_lam}. Importantly, the algorithms for selection of $g$'s and $\lambda$'s do not rely on any problem insight and do not require the choice of basis functions.
}

\paragraph{\new{Example (role of $\lambda$'s)}} To get some intuition to the interplay between $\ob$ and $\lambda$, consider the linear system $\dot{x} = ax$, $x\in [0,1]$ and the non-recurrent set $\Gamma = \{1\}$. In this case, any function on $\Gamma$ is constant, so the set of boundary functions $G$ can be chosen to consist of the constant function equal to one. Then it follows from~(\ref{eq:phi_def}) that given any complex number $\lambda \in \Cb$, the associated eigenfunction is $\phi_\lambda(x) := \phi_{\lambda,1}(x) = x^{\frac{\lambda}{a}}$. Given an observable $\ob$, the optimal choice of the set of eigenvalues $\Lambda\subset \Cb$ is such that the projection error~(\ref{opt:proj}) is minimized, which in this case translates to making
\begin{equation}\label{eq:projErrExample}
\min_{(c_\lambda\in\Cb)_{\lambda\in \Lambda}} \Big\|\ob -\sum_{\lambda\in\Lambda}c_\lambda x^{\frac{\lambda}{a}}\Big\|
\end{equation}
as small as possible with the choice of $\Lambda $. Clearly, the optimal choice (in terms of the number of eigenfunctions required to achieve a given projection error) of $\Lambda$ depends on $\ob$. For example, for $\ob = x$, the optimal choice is $\Lambda = \{a\}$, leading to a zero projection error with only one eigenfunction. For other observables, however, the choice of $\Lambda = \{a\}$ need not be optimal. For example, for $\ob = x^b$, $b\in \Rb$, the optimal choice leading to zero projection error with only one eigenfunction is $\Lambda = \{a\cdot b\}$. The statement of Theorem~\ref{thm:mainDensityThm} then translates to the statement that the projection error~(\ref{eq:projErrExample}) is zero for any $\Lambda = \{k\cdot \lambda_0\mid k\in \Nb_0 \}$ with $\lambda_0 < 0$ and any continuous observable $\ob$. The price to pay for this level of generality is the asymptotic nature of the result, requiring the cardinality of $\Lambda$ (and hence the number of eigenfunctions) going to infinity. %From a practical perspective of control and estimation, the number of eigenfunctions required to achieve a given projection error is of secondary importance (see Section~\ref{sec:denseForm}). From a theoretical perspective, this interesting question remains open and appears to be  related to the Prony method~\cite{plonka2014prony}.

\subsection{Generalized eigenfunctions}\label{sec:genEfun}
This section describes how \emph{generalized eigenfunctions} can be constructed with a simple modification of the proposed method. Importantly for this work, generalized eigenfunctions also give rise to Koopman invariant subspaces and therefore can be readily used for linear prediction. Given a complex number $\lambda$ and $g_1,\ldots, g_{n_\lambda}$, consider the Jordan block
\[
 J_\lambda =  \begin{bmatrix}
    \lambda & 1 & \\
    & \ddots & 1 \\
    & & \lambda
  \end{bmatrix}\, 
\]

and define

\begin{equation}\label{eq:genEig_def}
\begin{bmatrix}
\psi_{\lambda,g_1}(S_t(x_0)) \\ \vdots \\ \psi_{\lambda,g_{n_\lambda}}(S_t(x_0))
\end{bmatrix} = e^{J_\lambda t} \begin{bmatrix}g_1(x_0)\\\vdots\\ g_{n_\lambda}(x_0) \end{bmatrix}  
\end{equation}
for all $x_0\in\Gamma$ or equivalently
\begin{equation}\label{eq:genEig_def_1}
\begin{bmatrix}
\psi_{\lambda,g_1}(x) \\ \vdots \\ \psi_{\lambda,g_{n_\lambda}}(x)
\end{bmatrix} = e^{-J_\lambda \tau(x)} \begin{bmatrix}g_1(S_{\tau(x)}(x))\\\vdots\\ g_{n_\lambda}(S_{\tau(x)}(x)) \end{bmatrix}  
\end{equation}
for all  $x\in X_T$. Define also
\[
\bs\psi = \begin{bmatrix}\psi_{\lambda,g_1} ,   \ldots ,  \psi_{\lambda,g_{n_\lambda}}
\end{bmatrix}^\top.
\]
With this notation, we have the following theorem:
\begin{theorem}\label{thm:genEig}
Let $\Gamma$ be a non-recurrent set, $g_i\in \C(\Gamma)$, $i=1,\ldots,n_\lambda$, and $\lambda \in \Cb$. Then the subspace \[\mr{span}\{\psi_{\lambda,g_1},\ldots,\psi_{\lambda,g_{n_\lambda}}\} \subset \C(X_T)\] is invariant under the action of the Koopman semigroup $\Kc_t$. Moreover
\begin{equation}\label{eq:bsPsi_eq}
\bs\psi(S_t(x)) = e^{J_\lambda t}\bs\psi(x)
\end{equation}
and
\begin{equation}\label{eq:bsPsi_eq_diff}
\frac{d}{dt}\bs \psi(S_t(x)) = J_\lambda\bs\psi(S_t(x))
\end{equation}
for any $x \in X_T$ and any $t\in[0,T]$ such that $S_{t'}(x) \in X_T $ for all $t' \in [0,t]$.
\end{theorem}
\begin{proof}
Let $h \in \mr{span}\{\psi_{\lambda,g_1},\ldots,\psi_{\lambda,g_{n_\lambda}}\} $. Then $h = c^\top \bs \psi$ for some $c\in \Cb^{n_\lambda}$.
Given $x \in X_T$, we have
\[
\bs\psi(S_t(x)) = e^{-J_\lambda(\tau(x)-t)} \begin{bmatrix}g_1(S_{\tau(x)-t}(S_t(x)))\\\vdots\\ g_{n_\lambda}(S_{\tau(x)-t}(S_t(x))) \end{bmatrix} =  e^{J_\lambda t } e^{-J_\lambda(\tau(x))} \begin{bmatrix}g_1(S_{\tau(x)}(x))\\\vdots\\ g_{n_\lambda}(S_{\tau(x)}(x)) \end{bmatrix} =e^{J_\lambda t } \bs\psi(x),
\]
which is~(\ref{eq:bsPsi_eq}).
Therefore $\Kc_t h = c^\top\bs\psi\circ S_t = c^\top e^{tJ_\lambda} \bs\psi  \in \mr{span}\{\psi_{\lambda,g_1},\ldots,\psi_{\lambda,g_{n_\lambda}}\} $ as desired. Eq.~(\ref{eq:bsPsi_eq_diff}) follows immediately from~(\ref{eq:bsPsi_eq}).
\end{proof}

\paragraph{Beyond Jordan blocks}
The proof of Theorem~\ref{thm:genEig} reveals that there was nothing special of using a Jordan block in~(\ref{eq:genEig_def}). Indeed, the entire construction works with an \emph{arbitrary} matrix $A$ in place of~$J_\lambda$. However, nothing is gained by using an arbitrary matrix $A$ since the span of the generalized eigenfunctions constructed using $A$ is identical to that of the corresponding Jordan normal form of $A$, which is just the direct sum of the spans associated to the individual Jordan blocks.

\section{Learning eigenfunctions from data}\label{sec:learning}
\new{Now we use the construction of Section~\ref{sec:nonRecur} to learn eigenfunction from data. In particular, we leverage the freedom in \emph{choosing} (optimally, if possible) the eigenvalues $\lambda$ as well as the boundary functions $g$ to learn a \emph{rich} set of eigenfunctions such that the projection error~(\ref{opt:proj}) (and thereby also the prediction error~(\ref{eq:predError})) is minimized. The choice of $g$'s and $\lambda$'s is carried out using numerical optimization (convex in the case of $g$'s), without relying on problem insight and without requiring a choice of basis. We first describe how the eigenfunctions can be constructed from data, assuming the eigenvalues and boundary functions have been chosen and after that we describe how this choice can be carried out using numerical optimization.
}

Throughout this section we assume that we have available data in the form of $M_{\mr{t}}$ distinct equidistantly sampled trajectories with $M_{\mr{s}} +1$ samples each, where  $M_{\mr{s}} = T/ T_{\mr{s}} $ with $T_{\mr{s}}$ being the sampling interval (what follows straightforwardly generalizes to non-equidistantly sampled trajectories of unequal length). That is, the data is of the form
\begin{equation}\label{eq:data_uncont}
\D = \Big( (x_k^j)_{k=0}^{M_{\mr{s}}}  \Big)_{j=1}^{M_{\mr{t}}},
\end{equation}
where the superscript indexes the trajectories and the subscript the discrete time within the trajectory, i.e., $x_k^j = S_{kT_{\mr{s}}}(x_0^j)$, where $x_0^j$ is the initial condition of the $j^{\mr{th}}$ trajectory. The non-recurrent set $\Gamma$ is simply defined as
\[
\Gamma = \{x_0^1,\ldots, x_0^{M_{\mr{t}}}\}.
\]

We also assume that we have chosen a vector of complex numbers (i.e., the eigenvalues)
\[
\Lambda = (\lambda_1,\ldots,\lambda_{N})
\]
as well as a vector of continuous functions
\[
G = (g_1,\ldots, g_{N})
\]
defining the values of the eigenfunctions on the non-recurrent set $\Gamma$. The functions in $G$ will be referred to as \emph{boundary functions}.

Now we can construct $N$ eigenfunctions using the developments of Section~\ref{sec:nonRecur}. To do so, define the matrix
\begin{equation}\label{eq:G}
\mathbf{G}(i,j) = g_i(x_0^j)
\end{equation}
collecting the values of the boundary functions on the initial points of the trajectories in the data set. Define also \[
\phi_{\lambda_i,g_i}(x_0^j) := g_i(x_0^j),\quad j=1,\ldots, M_{\mr{t}}.
\]
Then Eq.~(\ref{eq:phi_comp_def}) uniquely defines the values of $\phi_{\lambda_i,g_i}$ on the entire data set. Specifically, we have
\begin{equation}\label{eq:eigFun_computed}
\phi_{\lambda_i,g_i}(x_k^j) = e^{\lambda_i k T_{\mr{s}}} \mathbf{G}(i,j)
\end{equation}
for all $k \in \{0,\ldots M_{\mr{s}} \}$ and all $j \in \{1,\ldots M_{\mr{t}} \}$. According to Lemma~\ref{lem:nonRecurSurface} (provided its assumptions hold), there exists an entire non-recurrent surface $\hat\Gamma$ passing through the initial conditions of the trajectories in the data set $\D$. Even though this surface is unknown to us, its existence implies that the eigenfunctions computed through~(\ref{eq:eigFun_computed}) on $\D$ are in fact samples of continuous eigenfunctions defined on \begin{equation}\label{eq:hatXT}
\hat X_T = \bigcup_{t\in [0,T]} S_t(\hat\Gamma)\,;
\end{equation}
see Figure~\ref{fig:nonRecur_data} for an illustration.
\begin{figure*}[tb]
	\begin{picture}(50,155)
		\put(80,0){\includegraphics[width=130mm]{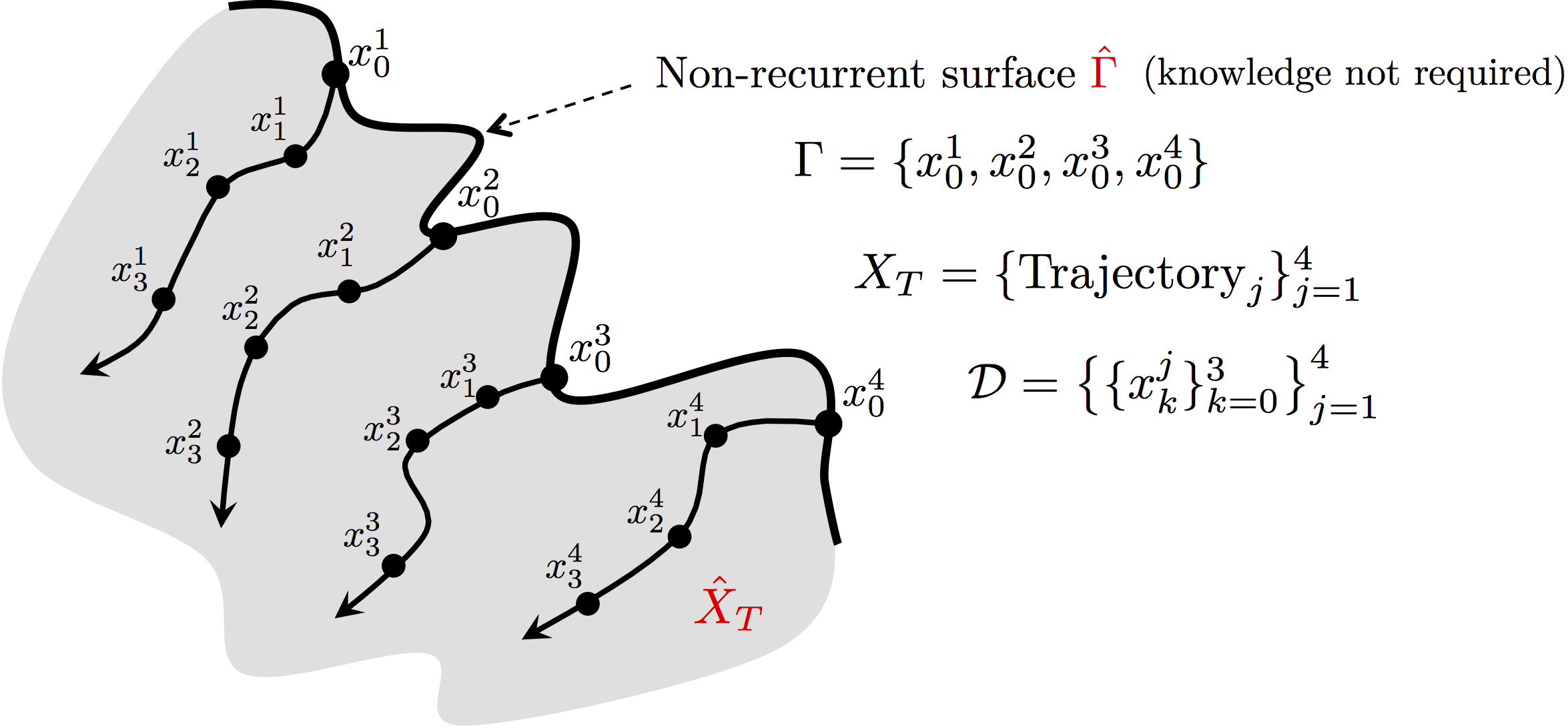}}
		%\put(348,138){\scriptsize (not known in advance)}
	\end{picture}
	\caption{\footnotesize Illustration of the non-recurrent set $\Gamma$, the non-recurrent surface $\hat\Gamma$. Note that the non-recurrent surface $\hat\Gamma$ does not need to be known explicitly for learning the eigenfunctions. Only sampled trajectories~$\D$ with initial conditions belonging to distinct trajectories are required (see Lemma~\ref{lem:nonRecurSurface}). Note also that even though the existence of the non-recurrent surface is assured by Lemma~\ref{lem:nonRecurSurface}, this surface can be highly irregular (e.g., oscillatory), depending on the interplay between the dynamics and the locations of the initial conditions.}
	\label{fig:nonRecur_data}
\end{figure*}
As a result, the eigenfunctions $\phi_{\lambda_i,g_i}$ can be learned on the entire set $\hat X_T $ (or possibly even larger region) via interpolation or approximation. Specifically, given a set of basis functions 
\[
\boldsymbol \beta = \begin{bmatrix} \beta_1\\\vdots\\ \beta_{N_{\beta}}\end{bmatrix}
\]
with $\beta_i \in \C(X)$, we can solve the interpolation problems 
\begin{equation}\label{opt:interp}
\begin{array}{ll}
\underset{c \in \Cb^{N_\beta}}{\mbox{minimize}} & \delta_1 \| c \|_1 + \|c\|_2^2  \\
\mbox{subject to} &  c^\top\boldsymbol \beta (x_k^j) = \phi_{\lambda_i,g_i}(x_k^j),\\
& k\in \{0,\ldots,M_{\mr{s}}\},\;\; j\in \{1,\ldots,M_{\mr{t}}\} \\
\end{array}
\end{equation}
for each $i\in\{1,\ldots,N\}$. Alternatively, we can solve the approximation problems
\begin{equation}\label{opt:regression}
\begin{array}{ll}
\underset{c \in \Cb^{N_\beta}}{\mbox{minimize}} &  \sum_{k=0}^{\mr{M_{\mr{s}}}} \sum_{j=1}^{M_\mr{t}}
\big| c^\top\bs\beta (x_k^j) - \phi_{\lambda_i,g_i}(x_k^j)\big|^2 +    \delta_1\| c \|_1 + \delta_2 \| c\|_2^2 \\
\end{array}
\end{equation}
for each $i\in\{1,\ldots,N\}$. In both problems the classical $\ell_1$ and $\ell_2$ regularizations are optional, for promoting sparsity of the resulting approximation and preventing overfitting; the numbers $\delta_1 \ge 0$, $\delta_2 \ge 0$ are the corresponding regularization parameters. The resulting approximation to the eigenfunction $\phi_{\lambda_i,g_i}$, denoted by $\hat\phi_{\lambda_i,g_i}$, is given by
\begin{equation}
\hat\phi_{\lambda_i,g_i}(x) = c_{i}^\top \bs\beta(x),
\end{equation}
where $c_{i}^\top $ is the solution to~(\ref{opt:interp}) or (\ref{opt:regression}) for a given $i\in\{1,\ldots,N\}$.  Note that the approximation $\hat\phi_{\lambda_i,g_i}(x)$ is defined on the entire state space $X$; if the interpolation method~(\ref{opt:interp}) is used then the approximation is exact on the data set $\D$ and one expects it to be accurate on $X_T$ and possibly also on $\hat X_T$, provided that the non-recurrent surface $\hat \Gamma$ (if it exists) and the functions in $G$ give rise to eigenfunctions well approximable (or learnable) by functions from the set of basis functions $\bs\beta$.  The eigenfunction learning procedure is summarized in Algorithm~\ref{alg:eigFunLerning}.

\begin{algorithm}
\caption{Eigenfunction learning}\label{alg:eigFunLerning}
\begin{algorithmic}[1]
\Require Data $\D = \Big( (x_k^j)_{k=0}^{M_{\mr{s}}}  \Big)_{j=1}^{M_{\mr{t}}}$, matrix $\mathbf{G} \in \Cb^{N\times M_{\mr{t}}}$, complex numbers $\Lambda=(\lambda_1,\ldots,\lambda_{N_\Lambda} )$, basis functions $\bs\beta = [\beta_1,\ldots,\beta_{N_\beta}]^\top$, sampling time $T_\mr{s}$.
\For{$i\in \{1,\ldots,N\}$}
 \For{$j \in \{1,\ldots,M_{\mr{t}}\}$, $k \in \{0,\ldots,M_{\mr{s}}\}$}
 \State $\phi_{\lambda,g}(x_k^j) := e^{\lambda k T_{\mr{s}}} \mathbf{G}(i,j)$
\EndFor
\State Solve~(\ref{opt:interp}) or (\ref{opt:regression}) to get $c_{i}$
\State Set $\hat\phi_{i} := c_{i}^\top \bs\beta$ 
\EndFor
\Ensure $\{\hat \phi_{i}\}_{i\in\{1,\ldots,N\}}$
\end{algorithmic}
\end{algorithm}

\new{\paragraph{Interpolation methods} We note that~(\ref{opt:interp}) and (\ref{opt:regression}) are only two possibilities for extending the values of the eigenfunctions from the data points to all of $X$; more sophisticated interpolation/approximation methods may yield superior results.}

\paragraph{Choosing initial conditions} As long as the initial conditions in $\Gamma$ lie on distinct trajectories that are non-periodic over the simulated time interval, the set $\Gamma$ is non-recurrent as required by our approach. This is achieved with probability one if, for example, the initial conditions are sampled uniformly at random over $X$ (assuming the cardinality of $X$ is infinite) and the dynamics is non-periodic or the simulation time is chosen such that the trajectories are non-periodic over the simulated time interval. In practice, one will typically choose the initial conditions such that the trajectories sufficiently cover a subset of the state space of interest (e.g., the safe region of operation of a vehicle). In addition, it is advantageous (but not necessary) to sample the initial conditions from a sufficiently regular surface (e.g., a ball or ellipsoid) approximating a non-recurrent surface in order to ensure that the resulting eigenfunctions are well behaved (e.g., in terms of the Lipschitz constant) and hence easily interpolable / approximable.

\new{\subsection{Optimal selection of boundary functions $g$}\label{sec:optimChoice_g} 
First we describe the general idea behind the optimal selection of $g$'s and then derive from it a convex-optimization-based algorithm. Given $\lambda \in \Cb$, let $\mathcal{L}_\lambda : \C(\Gamma)\to \C(X_T)$ denote the operator that maps a boundary function $g \in \C(\Gamma)$ to the associated eigenfunction $\phi_{\lambda,g}$, i.e.,
\[
\mathcal{L}_\lambda g = e^{-\lambda\tau} (g\circ S_{\tau} ).
\]
Notice, crucially, that this operator is linear in $g$. Given a vector of continuous boundary functions
\[
G = (g_1,\ldots, g_N),
\]
$g_i \in \C(\Gamma)$, and a vector of eigenvalues
\[
\Lambda = (\lambda_1,\ldots,\lambda_N),
\]
$\lambda_i\in \Cb$, the projection error~(\ref{eq:predError}) boils down to
\begin{equation}\label{eq:projErr_Vg}
\| \ob - \mr{proj}_{\mathcal{V}_{G,\Lambda} } \ob     \|,
\end{equation}
where
\[
\mathcal{V}_{G,\Lambda} = \mr{span}\{\Lc_{\lambda_1}g_1,\ldots,\Lc_{\lambda_N}g_N \}.
\]
Here $\mr{proj}_{\mathcal{V}_{G,\Lambda} }$ denotes the projection operator\footnote{\new{Depending on the norm used in~(\ref{eq:projErr_Vg}), the projection on $\mathcal{V}_{G,\Lambda}$ may not be unique, in which case we assume that a tiebreaker function has been applied, making the projection operator well defined.}} onto the finite-dimensional subspace $\mathcal{V}_{G,\Lambda}$.

The goal  is then to minimize~(\ref{eq:projErr_Vg}) with respect to $G$. In general, this is a non-convex problem. However, we will show that if each component of $\ob$ is considered separately, this problem admits a convex reformulation. Assume therefore that the total budget of~$N$ boundary functions is partitioned as
\[
G = (G_1,\ldots, G_{n_{\ob}} )
\]
with $\sum_{i=1}^{n_{\ob}}N_i = N$, where $N_i = \# G_i$ with $\# G_i$ denoting the cardinality of $G_i$. The eigenvalues are partitioned correspondingly, i.e., 
\[
\Lambda = (\Lambda_1,\ldots, \Lambda_{n_{\ob}} ),\quad \#\Lambda_i = N_i.
\]

Then, given $i \in \{1,\ldots, n_{\ob}\}$ and $\Lambda_i \in \Cb^{N_i}$, we want to solve
\[
\mathop{\mr{minimize}}\limits_{G_i\in \C(\Gamma)^{N_i}}\| \ob_i - \mr{proj}_{\mathcal{V}_{G_i,\Lambda_i} } \ob_i     \|.
\]
This problem is equivalent to
\[
\mathop{\mr{minimize}}\limits_{g_{i,j}\in\C(\Gamma),\;c_{i,j}\in \Cb}  \;  \big\| \ob_i - \sum_{j=1}^{N_i} c_{i,j} \mathcal{L}_{\lambda_{i,j}} g_{i,j}  \big\|.
\]
Since $\Lc_{\lambda}$ is linear in $g$,  we have $c_{i,j} \mathcal{L}_{\lambda_{i,j}} g_{i,j}  =  \mathcal{L}_{\lambda_{i,j}} (c_{i,j}g_{i,j}) $  with $c_{i,j}g_{i,j} \in \C(\Gamma)$ and hence, using the substitution $c_{i,j}g_{i,j} \leftarrow g_{i,j}  $, this problem is equivalent to
\begin{equation}\label{eq:projAbstract_final}
\mathop{\mr{minimize}}\limits_{g_{i,j}\in\C(\Gamma)}  \;  \big\| \ob_i - \sum_{j=1}^{N_i} \mathcal{L}_{\lambda_{i,j}} g_{i,j}  \big\|,
\end{equation}
which is a \emhp{convex} function of $g_{i,j}$ as desired.

\paragraph{Regularization} Beyond minimization of the projection error~(\ref{eq:projErr_Vg}), one may also wish to optimize the regularity of the resulting eigenfunctions $ \phi_{\lambda,g} = \mathcal{L}_{\lambda} g$ as these functions will have to be represented in a computer when working with data. Such regularity can be optimized by including an additional term penalizing, for instance, a norm of the gradient of $\mathcal{L}_\lambda g$, resulting in a Sobolev-type regularization. Adding this term to~(\ref{eq:projAbstract_final}) results in
\begin{equation}\label{eq:projAbstract_final_regul}
\mathop{\mr{minimize}}\limits_{g_{i,j}\in\C(\Gamma)}  \;  \big\| \ob_i - \sum_{j=1}^{N_i} \mathcal{L}_{\lambda_{i,j}} g_{i,j}  \big\| +\alpha\sum_{j=1}^{N_i} \|\D \mathcal{L}_{\lambda_{i,j}}g_{i,j}\|^{'},
\end{equation}
where $\mathcal{D}$ is a differential operator (e.g., the gradient), $\alpha\ge 0$ is a regularization parameter and the norms $\|\cdot\|$ and $\|\cdot\|'$ do not need to be the same.

\subsubsection{Data-driven algorithm for optimal selection of $g$'s}
Now we derive a data-driven algorithm from the abstract developments above. The idea is to optimize directly over the values of the boundary functions $g$, which, when restricted to the data set~(\ref{eq:data_uncont}), are just finitely many complex numbers collected in the matrix $\mathbf{G}$~(\ref{eq:G}). Formally, assume that the norm in (\ref{eq:projErr_Vg}) is given by the $L_2$ norm with respect to the empirical measure on the data set. In that case, problem~(\ref{eq:projAbstract_final}) becomes
\begin{equation}\label{eq:proj_discretized}
\mathop{\mr{minimize}}\limits_{\mathbf{g}_{i,j}\in\Cb^{M_{\mr{t}}}}  \;  \big\| \mathbf{h}_i - \sum_{j=1}^{N_i} \mathbf{L}_{\lambda_{i,j}} \mathbf{g}_{i,j}  \big\|_2,
\end{equation}
where $\mathbf{h}_i$ is the vector collecting the values of $\ob_i$ stacked on top each other trajectory-by-trajectory, the optimization variables $\mathbf{g}_{i,j} $ are vectors containing the values of the boundary functions $g_{i,j}$ on the starting points of the trajectories in the data set and the matrices $\mathbf{L}_{\lambda_{i,j}}$ are  given by
\[
\mathbf{L}_{\lambda_{i,j}} = \mr{bdiag(  \underbrace{\Lambda_{i,j},\ldots, \Lambda_{i,j}}_{M_\mr{t}\,\mr{times}}  )},
\]
where
\[
\Lambda_{i,j} = [ 1, e^{ \lambda_{i,j} T_\mr{s} }, e^{2\lambda_{i,j} T_\mr{s}}, \ldots, e^{ M_{\mr{s}}\lambda_{i,j} T_\mr{s} } ]^\top.
\]
This problem is equivalent to
\begin{equation}\label{eq:proj_discretized_final}
\mathop{\mr{minimize}}\limits_{\mathbf{g}_{i}\in\Cb^{N_{i}M_{\mr{t}}}}  \;  \big\| \mathbf{h}_i -  \mathbf{L}_{\Lambda_i}  \mathbf{g}_{i}  \big\|_2^2,
\end{equation}
where
\begin{equation}\label{eq:L_and_g_struct}
\mathbf{L}_{\Lambda_i} = [\mathbf{L}_{\lambda_{i,1}}, \mathbf{L}_{\lambda_{i,2}},\ldots, \mathbf{L}_{\lambda_{i,N_{i}}}],\quad \mathbf{g}_i = [\mathbf{g}_{i,1}^\top,\mathbf{g}_{i,2}^\top,\ldots,\mathbf{g}_{i,N_{i}}^\top]^\top.
\end{equation}
Problem~(\ref{eq:proj_discretized_final}) is a least-squares problem with optimal solution
\begin{equation}\label{eq:gopt}
\mathbf{g}_i^\star = \mathbf{L}_{\Lambda_i}^\dagger \mathbf{h}_i.
\end{equation}

%For this we assume that the norm used in~(\ref{eq:projErr_Vg}) is $L_2(X_T,\mu)$, where $\mu(A) = \int_{0}^T \nu(S_t^{-1}(A))\,dt  $.

Define the  $N$-by-$M_{\mr{t}}$ matrix   $\mathbf{G}$ by
\begin{equation}\label{eq:G_data_optim}
\mathbf{G} = \begin{bmatrix}
\mathbf{g}_{1,1}^\star &
\ldots & \mathbf{g}_{1,N_{1}}^\star&
\ldots& 
\mathbf{g}_{n_{\ob},1}^\star& 
\ldots &
\mathbf{g}_{n_{\ob},N_{{\ob}}}^\star
\end{bmatrix}^\top
\end{equation}
collecting the values of all boundary functions on the initial points of the trajectories, i.e., $\mathbf{G}(i,j)$ is the value of boundary function $i$ on the data point $x_0^j$ (here the vector $\mathbf{g}_{i}^\star$ is assumed to be partitioned in the same way as $\mathbf{g}_i$ in~(\ref{eq:L_and_g_struct})). The matrix $\mathbf{G}$ is then used in Algorithm~\ref{alg:eigFunLerning}.

\paragraph{Regularization} With regularization, assuming the $\|\cdot\|'$ norm in~(\ref{eq:projAbstract_final_regul}) is the $L_2$ norm with respect to the empirical measure on the data set, we get
\begin{equation}\label{eq:proj_discretized_final_reg}
\mathop{\mr{minimize}}\limits_{\mathbf{g}_{i}\in\Cb^{N_{i}M_{\mr{t}}}}  \;  \big\| \mathbf{h}_i -  \mathbf{L}_{\Lambda_i}  \mathbf{g}_{i}  \big\|_2^2 + \|   \mathbf{D}\mathbf{L}_{\Lambda_i}\mathbf{g}_i    \|_2^2,
\end{equation}
where $\mathbf{D}$ is a discrete representation of the differential operator $\mathcal{D}$ used in~(\ref{eq:projAbstract_final_regul}). The optimal solution to~(\ref{eq:proj_discretized_final_reg}) is given by
\[
\mathbf{g}_i^\star = \begin{bmatrix}
\mathbf{L}_{\Lambda_i} \\ \mathbf{D}\mathbf{L}_{\Lambda_i}
\end{bmatrix}^\dagger \begin{bmatrix}
\mathbf{h}_i \\ 0
\end{bmatrix}.
\]
The matrix $\mathbf{G}$ is then defined by~(\ref{eq:G_data_optim}) and used in Algorithm~\ref{alg:eigFunLerning}.

%\mathbf{G} = \begin{bmatrix}
%\mathbf{g}_{1,1}^\top \\
%\vdots\\
%\mathbf{g}_{1,N_{G_i}}^\top\\
%\vdots \\
%\mathbf{g}_{n_{\ob},1}^\top \\
%\vdots\\
%\mathbf{g}_{n_{\ob},N_{G_{\ob}}}^\top
%\end{bmatrix}

}

\new{\subsection{Selection of eigenvalues $\lambda$}\label{sec:optimChoice_lam}
Now we describe how to select the eigenvalues using numerical optimization. For simplicity, we will work in the setting without regularization, the generalization being straightforward. Plugging in~(\ref{eq:gopt}) into~(\ref{eq:proj_discretized_final}) and using the fact that $\mathbf{L}_{\Lambda_i}\mathbf{L}_{\Lambda_i}^\dagger$ is the orthogonal projection operator onto the column space of $\mathbf{L}_{\Lambda_i}$ (and hence is self-adjoint and idempotent), the minimum in~(\ref{eq:proj_discretized_final}) is equal to
\begin{equation}\label{eq:lamObj}
\|\mathbf{h}_i\|_2^2 - \| \mathbf{L}_{\Lambda_i}\mathbf{L}_{\Lambda_i}^\dagger \mathbf{h}_i \|_2^2.
\end{equation}
The optimal choice of $\Lambda_i = (\lambda_{i,1},\ldots,\lambda_{i,N_i})\in \Cb^{N_i}$ minimizes~(\ref{eq:lamObj}). Unfortunately, (\ref{eq:lamObj}) is a non-convex function of $\Lambda_i$ with no obvious convexification available. Fortunately, the value of $N_i$ is typically modest and therefore this optimization problem can be (at least approximately) solved by using local optimization initialized from a collection of randomly chosen initial conditions. Crucial to this is the availability of an analytic expression for the gradient of~(\ref{eq:lamObj}) with respect to $\Lambda_i$; for simplicity of analysis we assume that the matrix $\mathbf{L}_{\Lambda_i}\herm\mathbf{L}_{\Lambda_i}$ is invertible, which is the case generically provided that $M_{\mr{t}}M_{\mr{s}} > M_{\mr{t}} N_i$ (in which case the matrix $\mathbf{L}_{\Lambda_i}$ is tall). Deriving the gradient in the fully general setting is straightforward but lengthy, using the derivative of matrix pseudoinverse~\cite{stewart1977perturbation}. Assuming $\mathbf{L}_{\Lambda_i}\herm\mathbf{L}_{\Lambda_i}$  is invertible, (\ref{eq:lamObj}) becomes
\begin{equation}\label{eq:p_simplified}
p_i(\Lambda_i) := \|\mathbf{h}_i\|^2_2 - \mathbf{h}_i\herm \mathbf{L}_{\Lambda_i}(\mathbf{L}_{\Lambda_i}\herm\mathbf{L}_{\Lambda_i})^{-1}\mathbf{L}_{\Lambda_i}\herm \mathbf{h}_i.
\end{equation}
Using the fact that for a matrix $A(\theta)$, depending on $\theta\in \Rb$, we have \[
\frac{d}{d\theta} [A(\theta)^{-1}] = -A(\theta)^{-1} \frac{dA(\theta)}{d\theta} A(\theta)^{-1},\] 
we obtain
\begin{equation}\label{eq:p_grad}
\frac{\partial{p_i(\Lambda_{i,j}) }}{\partial \theta_{i,j}} = -2\R\left\{\mathbf{h}_i\herm\frac{\partial \mathbf{L}_{\Lambda_i}}{\partial \theta_{i,j}} \mathbf{q}_i\right\} + \mathbf{q}_i\herm \left[  \mathbf{L}_{\Lambda_i}\herm   \frac{\partial \mathbf{L}_{\Lambda_i}}{\partial \theta_{i,j}} + \left(\frac{\partial \mathbf{L}_{\Lambda_i}}{\partial\theta_{i,j}}\right)\herm  \mathbf{L}_{\Lambda_i}      \right]\mathbf{q}_i,
\end{equation}
where $\theta_{i,j} \in \Rb$ stands either for the real or imaginary part of $\lambda_{i,j}$ (the expressions are the same for both) and
\[
\mathbf{q}_i = (\mathbf{L}_{\Lambda_i}\herm\mathbf{L}_{\Lambda_i} )^{-1}\mathbf{L}_{\Lambda_i}\herm \mathbf{h}_i.
\]
Continuing, we have
\[
\frac{\partial \mathbf{L}_{\Lambda_i}}{\partial \theta_{i,j}} = \left[0,\ldots,0, \frac{\mathbf{L}_{i,j}}{\partial \theta_{i,j}}, 0,\ldots,0 \right]
\]
with the same block structure as in~(\ref{eq:L_and_g_struct}) and the non-zero element on the $j^{\mr{th}}$ block position and with
\[
\frac{\mathbf{L}_{i,j}}{\partial \theta_{i,j}} = \mr{bdiag}\left(\frac{\partial \Lambda_{i,j}}{\partial \theta_{i,j}},\ldots,\frac{\partial \Lambda_{i,j}}{\partial \theta_{i,j}}\right),
\]
where
\[
\frac{\partial \Lambda_{i,j}}{\partial \lambda_{i,j}^\Rb} = T_{\mr{s}}\begin{bmatrix}0\\1\\2\\ \vdots \\M
_{\mr{s}}  \end{bmatrix}\circ \Lambda_{i,j},\quad \frac{\partial \Lambda_{i,j}}{\partial \lambda_{i,j}^\Ib} = \sqrt{-1}\frac{\partial \Lambda_{i,j}}{\partial \lambda_{i,j}^\Rb}, 
\]
with $\circ$ denoting the componentwise (Hadamard) product and $\sqrt{-1}$  the imaginary unit and with $\lambda_{i,j}^\Rb$ and $\lambda_{i,j}^\Ib$ denoting the real respectively imaginary parts of $\lambda_{i,j}$.

This allows us to evaluate $\frac{\partial{p}(\Lambda_i)}{\partial \lambda_{i,j}^\Rb}$ and $\frac{\partial{p}(\Lambda_i)}{\partial\lambda_{i,j}^\Ib}$ for all $j$ and hence allows us to evaluate the gradient of $p(\Lambda_i)$ with respect to the real and imaginary parts of the eigenvalues in $\Lambda_i$.

\paragraph{Special structure} We note that without regularization, the least-squares problem~(\ref{eq:proj_discretized_final}) can be decomposed ``trajectory-by-trajectory'' to $M_{\mr{t}}$ independent least-squares problems. Moreover, the matrix in each of these least-squares problems is a Vandermonde matrix for which specialized least-squares solution methods exist (e.g.,~\cite{drmavc2018least}). This special special structure also implies that the gradient of $p(\Lambda_i)$ is a sum of $M_{\mr{t}}$ independently computable terms and hence amenable to parallel computation.
 }

\subsection{Generalized eigenfunctions from data} Algorithm~\ref{alg:eigFunLerning} can be readily extended to the case of generalized eigenfunctions as described in Section~\ref{sec:genEfun}. Step~3 of this algorithm is replaced by
\[
\begin{bmatrix}
\psi_{\lambda,g_1}(x_k^j) \\ \vdots \\ \psi_{\lambda,g_{n_\lambda}}(x_k^j)
\end{bmatrix} = e^{J_\lambda k T_{\mr{s}}} \begin{bmatrix}g_1(x_0^j)\\\vdots\\ g_{n_\lambda}(x_0^j) \end{bmatrix},  
\]
where, as in Section~\ref{sec:genEfun}, $J_\lambda$ is a Jordan block of size $n_\lambda$ associated to an eigenvalue $\lambda$ and $g_1,\ldots,g_{n_\lambda}$ are continuous boundary functions. Step~4 of Algorithm~\ref{alg:eigFunLerning} (interpolation / approximation) is then performed on each $\psi_{\lambda,g_i}$ separately. Note that with Jordan block of size one, the entire procedure reduces to the case of eigenfunctions.

We note that there are no restrictions on the parings of Jordan blocks (of arbitrary size) and continuous boundary functions, thereby providing additional freedom for constructing very rich invariant subspaces of the Koopman operator.

\subsection{Obtaining matrices $A$ and $C$}
Here we describe how to obtain the matrices $A$ and $C$ in~~(\ref{eq:linPred_uncont}).
Let
\[
 \bs{\hat\phi} = \begin{bmatrix}\hat\phi_1\\\vdots\\\hat\phi_N
\end{bmatrix}
\]
denote the vector of $N$ eigenfunction approximations obtained from Algorithm~\ref{alg:eigFunLerning}. \new{The matrix $A$ is then given simply by
\[
A = \mr{diag}(\lambda_1,\ldots,\lambda_N),
\]
where $\lambda_i$ are eigenvalues associated to $\hat\phi_i$. Provided that the boundary functions were chosen optimally as described in Section~\ref{sec:optimChoice_g}, the matrix $C$ is obtained as
\begin{equation}\label{eq:Cbdiag}
C = \mr{bdiag}( \mathbf{1}_{N_1},\ldots,\mathbf{1}_{N_{n_{\ob}}} ),
\end{equation}
where $\mathbf{1}_{N_i}$ is a row vector of ones of length $N_i$ and $N_i$ constitute a partition of $N$ as described in Section~\ref{sec:optimChoice_g}.}

Generally, irrespective of how the boundary functions were chosen, the matrix $C$ can be obtained by (approximately) solving (\ref{opt:proj}) with $\bs \phi$ replaced by $\bs{\hat\phi} $. This problem is typically not solvable analytically in high dimensions since it requires a multivariate integration or uniform bounding of a continuous function (depending on the norm used in~(\ref{opt:proj})). Therefore, we use a sample-based approximation.  If the $L_2$ norm is used in~(\ref{opt:proj}), we solve the optimization problem
\begin{equation}\label{opt:projL2}
\underset{C \in \Cb^{n_\ob\times N}}{\mbox{minimize}} \;\;\sum_{i=1}^{\bar M} \big \|\ob (\bar x_i) - C\bs{\hat\phi}(\bar x_i)\big\|^2_2.
\end{equation}
For the sup-norm, we solve
\begin{equation}\label{opt:projSup}
\underset{C \in \Cb^{n_\ob\times N}}{\mbox{minimize}}\;\; \max_{i\in\{1,\ldots,\bar M\}} \big \|\ob (\bar x_i) - C\bs{\hat\phi}(\bar x_i)\big \|_\infty.
\end{equation}
The samples $\{\bar x_i\}_{i=1}^{\bar M}$ can either coincide with the samples $\{x_k^j\}_{j,k}$ used for learning of the eigenfunctions or they can be generated anew (e.g., to emphasize certain regions of state-space where accurate projection (and hence prediction) is required). See Section~\ref{sec:compAspects} for a discussion of computational aspects of solving these two problems.

\subsection{Exploiting algebraic structure}\label{sec:algStructure}
It follows immediately from~(\ref{eq:efun_def}) that products and powers of eigenfunctions are also eigenfunctions. In particular, given $\phi_1,\ldots,\phi_{N_0}$ eigenfunctions with the associated eigenvalues $\lambda_1,\ldots,\lambda_{N_0}$, the function
\begin{equation}\label{eq:efunAlgStruct}
\phi = \phi_1^{p_1}\cdot \ldots \cdot \phi_{N}^{p_{N_0}}
\end{equation}
is a Koopman eigenfunction with the associated eigenvalue
\[
\lambda = p_1\lambda_1 + \ldots + p_{N_0}\lambda_{N_0}.
\]
This holds for any nonnegative \emph{real or integer} powers $p_1,\ldots,p_{N_0}$.

This algebraic structure can be exploited to generate additional eigenfunction approximations starting from those obtained using Algorithm~\ref{alg:eigFunLerning}, at a very little additional computational cost. In particular, one can construct only a handful of eigenfunction approximations using Algorithm~\ref{alg:eigFunLerning}, e.g., with $\Lambda $ being a single real eigenvalue or a single complex conjugate pair and the set $G$ consisting of linear coordinate functions $x_i$, $i=1,\ldots, n$. This initial set of eigenfunctions can then be used to generate a very large number of additional eigenfunction approximations using~(\ref{eq:efunAlgStruct}) in order to ensure that the projection error~(\ref{opt:proj}) is small. When queried at a previously unseen state (e.g., during feedback control), only the eigenfunction approximations $\hat\phi_1,\ldots,\hat\phi_{N_0}$ have to be computed using interpolation or approximation (which can be costly if the number of basis functions $\bs\beta$ is large in step 5 of Algorithm~\ref{alg:eigFunLerning}) whereas the remaining eigenfunction approximations are obtained by simply taking powers and products according to~(\ref{eq:efunAlgStruct}).

\subsection{Computational aspects}\label{sec:compAspects}
The main computational burden of the proposed method is the solution to the interpolation or approximation problems~(\ref{opt:interp}) and (\ref{opt:regression}). Both these problems are \emph{convex} optimziation problems that can be reliably solved using generic packages such as MOSEK or Gurobi. For very large problem instances, specialized packages for $\ell_1$ / $\ell_2$ regularized least-squares problems may need to be deployed (see, e.g. \cite{yang2013fast,parikh2014proximal}). We note that for each pair $(\lambda,g)$ the coefficients $\bs\beta(x_k^j)$ remain the same, which can be exploited to drastically speed up the solution.

For problems without $\ell_1$ regularization, we have an explicit solution
\[
c_{i} = \bf W^\dagger \bf w
\]
 for~(\ref{opt:interp}) and
\[
c_{i} = ({\bf W} + \delta_2 {\bf I})^\dagger \bf w,
\]
for~(\ref{opt:regression})
where
\[
\bf W = \begin{bmatrix}
 \bs \beta (x_0^1) & \ldots & \bs\beta (x_{M_{\mr{s}}}^1) & \bs\beta (x_0^2) & \ldots & \bs \beta (x_{M_{\mr{s}}}^2) & \ldots & \bs\beta (x_0^{M_{\mr{t}}}) & \ldots & \bs \beta (x_{M_{\mr{s}}}^{M_{\mr{t}}})
\end{bmatrix}^\top
\]
and
\[
\bf w = \begin{bmatrix}
 \phi_{i} (x_0^1) & \ldots & \phi_{i}  (x_{M_{\mr{s}}}^1) & \phi_{i}  (x_0^2) & \ldots & \phi_{i}  (x_{M_{\mr{s}}}^2) & \ldots & \phi_{i}  (x_0^{M_{\mr{t}}}) & \ldots & \phi_{i}  (x_{M_{\mr{s}}}^{M_{\mr{t}}})
\end{bmatrix}^\top,
\]
where  $\phi_{i} = \phi_{\lambda_i,g_i} $ as defined in~(\ref{eq:eigFun_computed}).

The projection problems~(\ref{opt:projL2}) and (\ref{opt:projSup}) are both convex optimization problems that can be easily solved using generic convex optimization packages (e.g., MOSEK or Gurobi). The use of such tools is necessary for the sup-norm projection problem~(\ref{opt:projSup}). However, for the least-squares projection problem~(\ref{opt:projL2}), linear algebra is enough with the analytical solution being
\[
C = [\ob(\bar x_1),\ldots, \ob(\bar x_{\bar M})] [\bs{\hat\phi}(\bar x_1),\ldots, \bs{\hat\phi}(\bar x_{\bar M})] ^\dagger.
\]

\section{Linear predictors for controlled systems}\label{sec:cont}
In this section we describe how to build linear predictors for controlled systems. Assume a nonlinear controlled system of the form
\new{
\begin{equation}\label{eq:sys_cont}
	\new{\dot{x} =  f(x) + Hu}
\end{equation}
with the state $x\in X\subset \Rb^n$ and control input $u\in U\subset \Rb^m$ and $H\in \Rb^{n\times m}$. This form is general since any control system of the form $\dot{\tilde x} = \tilde{f}(\tilde x,v)$ can be transformed to~(\ref{eq:sys_cont}) using the state inflation\footnote{\new{Imposing constraints on the control input of the state-inflated system corresponds to imposing constraints on the derivative of the original control input, which is important in practical applications.}} $x = [\tilde x^\top,v^\top]^\top$, $\dot v = u$, which leads to

\[
f := \begin{bmatrix}
\tilde{f}(\tilde{x},v) \\ 0 \end{bmatrix} ,\quad H := \begin{bmatrix} 0 \\ \mathbf{I} \end{bmatrix}.
\]

}

 As in~\cite{korda2018linear}, the goal is to construct a predictor in the form of a controlled \emph{linear} dynamical system
\begin{subequations}\label{eq:linPred_cont}
\begin{align}
\dot{z} &=  Az + Bu \\
z_0 &= \bs {\hat \phi}(x_0), \\
\hat{y} &=  Cz.
\end{align}
\end{subequations}

Whereas~\cite{korda2018linear} uses a one-step procedure (essentially a generalization of the extended dynamic mode decomposition (EDMD) to controlled systems), here we follow a two-step procedure, where we first construct eigenfunctions for the uncontrolled system
\[
\dot{x} = f(x).
\]
We assume that we have two data sets available. The first one is an uncontrolled dataset $\D$ with the same structure as in~(\ref{eq:data_uncont}) in Section~\ref{sec:learning}. The second data set, $\D_{\mr{c}}$, is with control in the form of $M_{\mr{t},\mr{c}}$ equidistantly sampled trajectories with $M_{\mr{s},\mr{c}}+1$ samples each, i.e.,
\begin{equation}\label{eq:data_uncont}
\D_{\mr{c}} = \Big( (x_k^j)_{k=0}^{M_{\mr{s},\mr{c}}}, (u_k^j)_{k=0}^{M_{\mr{s},\mr{c}}-1}  \Big)_{j=1}^{M_{\mr{t},\mr{c}}},
\end{equation}
where $x_{k+1}^j = S_{T_{\mr{s}}}(x_k^j,u_k^j)$, where $S_t(x,u)$ denotes the solution to~(\ref{eq:sys_cont}) at time $t$ starting from $x$ and with the control input held constant and equal to $u$ in $[0,t]$. We note that both the number of trajectories and the trajectory length may differ for the controlled and uncontrolled data sets.

\paragraph{Step 1 -- $\bs {\hat \phi}$, $A$, $C$} In the first step of the procedure we construct approximate eigenfunctions $\bs {\hat \phi}$ of~(\ref{eq:sys}) (with $f(x) = f_{\mr{c}}(x,0)$) using the procedure described in Section~\ref{sec:learning} , obtaining also the matrices $A$ and $C$.

\paragraph{Step 2 -- matrix $B$} In order to obtain the matrix $B$ we perform a regression on the controlled data set~(\ref{eq:data_uncont}). The quantity to be minimized is a \emph{multi-step} prediction error. Crucially, this multi-step error can be minimized in a \emph{convex} fashion;  this is due to the fact that the matrices $A$ and $C$ are already known and fixed at this step and the predicted output $\hat y$ of~(\ref{eq:linPred_cont}) depends affinely on $B$. This is in stark contrast to EDMD-type methods, where only one-step ahead prediction error can be minimized in a convex fashion. In order to keep expressions simple we assume that the time interval over which we want to minimize the prediction error coincides with the length of the trajectories in our data set (everything generalizes straightforwardly to shorter prediction times). The problem to be solved therefore is
\begin{equation}\label{eq:multiStepError}
\underset{B_{\mr{d}}\in \Rb^{N\times m}}{\mbox{minimize}}\;\;\sum_{j=1}^{M_{\mr{t},\mr{c}}} \sum_{k=1}^{M_{\mr{s},\mr{c}}} \|\ob(x_k^j) - \hat y_k(x_0^j)   \|_2^2,
\end{equation}
where
\[
\hat y_k(x_0^j) = CA_{\mr{d}}^k z_{0}^j + \sum_{i = 0}^{k-1}  CA_{\mr{d}}^{k-i-1} B_{\mr{d}} u_{i}^j
\]
is the output $\hat y$ of~(\ref{eq:sys_cont}) at time $kT_{\mr{s}}$ starting from the (known) initial condition \[
z_0^j = \bs {\hat \phi}(x_0^j).
\]
The discretized matrices $A_{\mr{d}}$ (known) and $B_{\mr{d}}$ (to be determined) are related to $A$ and $B$  by
\begin{equation}\label{eq:AdBd}
A_{\mr{d}} = e^{ A T_{\mr{s}} } \;,\quad B_{\mr{d}} = \left(\int_0^{T_{\mr{s}}}e^{-As}\,ds\right) B.
\end{equation}
We note that in the above expression the matrix multiplying $B$ is invertible for any $T_{\mr{s}} > 0$ and therefore $B$ can be uniquely recovered from the knowledge of $B_{\mr{d}}$. Using vectorization, the output $\hat y_k(x_0^j) $ can be re-written as
\begin{equation}\label{eq:outputPred_cont}
\hat y_k(x_0^j) = CA_{\mr{d}}^k z_{0}^j + \sum_{i = 0}^{k-1}  \big[(u_{i}^j)^{\top} \otimes (CA_{\mr{d}}^{k-i-1})\big] \mr{vec}(B_{\mr{d}} ),
\end{equation}
where $\mr{vec}(\cdot)$ denotes the (column-major) vectorization of a matrix and $\otimes$ the Kronecker product. 
Since $A_{\mr{d}}$, $C$, $z_{0}^j$ and $\ob(x_k^j)$ are all known, plugging in~(\ref{eq:outputPred_cont}) to the least-squares problem~(\ref{eq:multiStepError}) leads to the minimization problem
\begin{equation}\label{opt:minMultistep}
\underset{b \, \in\, \Rb^{m N}}{\mbox{minimize}} \;\; \|  \Theta b - \theta \|_2^2,
\end{equation}
where
\[
\Theta = \begin{bmatrix}\Theta_1^\top & \Theta_2^\top & \ldots & \Theta_{M_{\mr{t}}}^\top \end{bmatrix}^\top , \qquad \theta = \begin{bmatrix}\theta_1^\top & \theta_2^\top & \ldots & \theta_{M_{\mr{t}}}^\top \end{bmatrix}^\top
\]
with
\[
\Theta_{j} = \begin{bmatrix}
(u_0^j)^\top \otimes C  \\ (u_1^j)^\top \otimes C + u_0^j\otimes(CA_{\mr{d}}) \\ \vdots \\
\sum_{i = 0}^{M_{\mr{s}}-1}  \big[(u_{i}^j)^{\top} \otimes (CA_{\mr{d}}^{M_{\mr{s}}-i-1})\big]
\end{bmatrix},\qquad \theta_{j} = \begin{bmatrix} \ob(x_1^j) - CA_{\mr{d}}z_0^j \\ \ob(x_2^j) - CA_{\mr{d}}^2z_0^j \\ \vdots \\ \ob(x_{M_{\mr{s}}}^j) - CA_{\mr{d}}^{M_{\mr{s}}}z_0^j 
\end{bmatrix}.
\]

The matrix $B_{\mr{d}}$ is then given by
\begin{equation}\label{eq:Bd}
B_{\mr{d}} = \mr{vec}^{-1}(\Theta^\dagger\theta),
\end{equation}
where $\Theta^\dagger\theta$ is an optimal solution to~(\ref{opt:minMultistep}). Since $A = \mr{diag}(\lambda_1,\ldots,\lambda_N)$, the matrix $B$ is obtained as
\begin{equation}\label{eq:B}
B = \left(\int_0^{T_{\mr{s}}}e^{-As}\,ds\right)^{\!\!-1} \!\!\!\!B_{\mr{d}}  = \mr{diag}\Big(\frac{\lambda_1}{1-e^{-\lambda_1T_\mr{s}}}\, ,\ldots\,,  \frac{\lambda_N}{1-e^{-\lambda_NT_\mr{s}}}\Big)B_{\mr{d}}.
\end{equation}

\subsection{Koopman model predictive control}\label{sec:KoopmanMPC}
In this section we briefly describe how the linear predictor~(\ref{eq:linPred_cont}) can be used within a \emph{linear} model predictive control (MPC) scheme to control \emph{nonlinear} dynamical systems. This method was originally developed in~\cite{korda2018linear} and this section closely follows this work; the reader is referred therein for additional details as well as to~\cite{korda2018power,arbabi2018data} for applications in power grid and fluid flow control.

An MPC controller solves at each step of a closed-loop operation an optimization problem where a given cost function is minimized over a finite prediction horizon with respect to the predicted control inputs and predicted outputs of the dynamical system. For nonlinear systems, this is almost always a nonconvex optimization problem due to the equality constraint in the form of the nonlinear dynamics. In the Koopman MPC framework, on the other hand, we solve the \emph{convex quadratic} optimization problem (QP)

%J\left((u_i)_{i=0}^{N_{\mr{p}}-1}, (\hat y_i)_{i=0}^{N_{\mr{p}}}  \right) 

\begin{equation}\label{eq:MPC}
\begin{array}{ll}
\underset{u_i, z_i, \hat y_i }{\mbox{minimize}} & z_{N_{\mr{p}}}^\top Q_{N_{\mr{p}}} z_{N_{\mr{p}}} + q_{N_{\mr{p}}}^\top z_{N_{\mr{p}}} + \sum_{i=0}^{N_{\mr{p}}-1} z_i^\top Q_i z_i + u_i^\top R_i u_i + q_i^\top z_i + r_i^\top u_i   \\
\mbox{subject to} & z_{i+1} = A_{\mr{d}}z_i + B_{\mr{d}}u_i, \quad  \hspace{3mm} i = 0,\ldots, N_{\mr{p}}-1 \\
& E_i \hat y_i + F_iu_i \le b_i,  
\hspace{4.5mm}\quad \quad i = 0,\ldots, N_{\mr{p}}-1\\
&E_{N_{\mr{p}}}\hat y_{N_{\mr{p}}} \le b_{N_{\mr{p}}}\\
& \hat y_i = C z_i\\
\mbox{parameter} & z_0 = {\bs {\hat \phi}}(x_{\mr{current}}) ,
\end{array}
\end{equation}
where the cost matrices $Q_i \in \Rb^{n_\ob\times n_\ob}$ and $R_i \in \Rb^{m\times m}$  are positive semidefinite and $N_{\mr{p}}$ is the prediction horizon. The optimization problem is parametrized by $x_{\mr{current}} \in \Rb^n$ which is the current state measured during the closed-loop operation. The control input applied to the system is the first element of the control sequence optimal in~(\ref{eq:MPC}). Notice that in~(\ref{eq:MPC}) we use directly the discretized predictor matrices $A_{\mr{d}}$ and $B_{\mr{d}}$, where $A_{\mr{d}} = \mr{diag}(e^{\lambda_1 T_{\mr{s}}},\ldots,e^{\lambda_N T_{\mr{s}}})$ and $B_{\mr{d}}$ is given by~(\ref{eq:Bd}) with $T_{\mr{s}}$ being the sampling interval. See Algorithm~\ref{alg:koopmanMPC} for a summary of the Koopman MPC in this sampled data setting.

\paragraph{Handling nonlinearities} Crucially, all nonlinearities in $x$ are subsumed in the output mapping $\ob$ and therefore predicted in a linear fashion through~(\ref{eq:linPred_cont}) (or its discretized equivalent). For example, assume we wish to minimize the predicted cost
\begin{equation}\label{eq:costNL}
J_{{\mr{nonlin}}} = J_{\mr{quad}} + l_{N_{\mr{p}}}(x_{N_{\mr{p}}}) + \sum_{i=0}^{N_{\mr{p}}-1} l(x_i),
\end{equation}
subject to the stage and terminal constraints
\begin{subequations}\label{eq:constNL}
\begin{align}
c(x_i) + Du &\le 0,  \qquad i = 0,\ldots, N_{\mr{p}} - 1, \\
c_{N_{\mr{p}}}(x_i) & \le 0, \qquad  i = 0,\ldots, N_{\mr{p}} - 1,
\end{align}
\end{subequations}
where
\[
J_{\mr{quad}} = x_{N_{\mr{p}}}^\top Q_{N_{\mr{p}}} x_{N_{\mr{p}}} + q_{N_{\mr{p}}}^\top x_{\mr{N}_p}  + \sum_{i=0}^{N_{\mr{p}}-1} x_i^\top Q x_i + u_i^\top R u_i +q^\top x_i +  r^\top u_i
\]
is convex quadratic and $c:\Rb^n\to \Rb^{n_c}$, $c_{N_{\mr{p}}}:\Rb^n\to \Rb^{n_{c_\mr{p}}}$, $l:\Rb^n
\to\Rb$ and $l_{N_{\mr{p}}}:\Rb^n\to\Rb$ are nonlinear functions. The mapping $\ob$ is then set to
\[
\ob (x) = \begin{bmatrix}
x \\ l(x) \\ l_{N_{\mr{p}}}(x) \\ c(x) \\ c_{N_{\mr{p}}}(x)
\end{bmatrix}.
\]
Replacing $\ob$ by $\hat y$ in~(\ref{eq:costNL}), the objective function $J_{{\mr{nonlin}}}$ translates to a \emph{convex} quadratic in $(u_0,\ldots,u_{N_{\mr{p}}-1})$ and $(\hat y_0,\ldots,
\hat y_{N_{\mr{p}}})$; similarly the stage and terminal constraints~(\ref{eq:constNL}) translate to \emph{affine} (and hence convex) inequality constraints on $(u_0,\ldots,u_{N_{\mr{p}}-1})$ and $(\hat y_0,\ldots,
\hat y_{N_{\mr{p}}})$.

Note that polytopic constraints on control inputs can be encoded by selecting certain components of the vector function $c(x_i)$ equal to constant functions. For example, box constraints on $u$ of the form $u\in [u_{\mr{min}},u_{\mr{max}}]$ with $u_{\mr{min}} \in \Rb^m$ and $u_{\mr{max}} \in \Rb^m$ are encoded by selecting
\[
c(x) = \begin{bmatrix}   -u_{\mr{max}} \\ u_{\mr{min}} \\ \tilde{c}(x) \end{bmatrix},\quad D = \begin{bmatrix}  {\bf I} \\ -{\bf I} \\ \tilde D \end{bmatrix},
\]
where $\tilde{c}(\cdot)$ and $\tilde D$ model additional state-input constraints.

\paragraph{No free lunch} At this stage it should be emphasized that since $\hat y$ is only an approximation of the true output~$\ob$, the convex QP~(\ref{eq:MPC}) is only an approximation of the nonlinear MPC problem with~(\ref{eq:costNL}) as objective, and (\ref{eq:sys_cont}) and (\ref{eq:constNL}) as constraints. This is unavoidable at this level of generality since such nonlinear MPC problems are typically NP-hard whereas the convex QP~(\ref{eq:MPC}) is polynomial time solvable. Nevertheless, as long as the prediction $\hat y$ is accurate, we also expect the solution of the linear MPC problem (\ref{eq:MPC}) to be close to the optimal solution of the nonlinear MPC problem, thereby resulting in near-optimal closed-loop performance.

\subsubsection{Dense form Koopman MPC}\label{sec:denseForm}
Importantly for real-world deployment, in problem~(\ref{eq:MPC}), the possibly high-dimensional variables $z_i$ and $\hat y_i$ can be solved for in terms of the variable
\[
{\bf u} := [u_0,\ldots,u_{N_{\mr{p}}-1}]^\top,
\]
obtaining
\begin{equation}\label{eq:MPC_dense}
\begin{array}{ll}
\underset{{\bf u} \in \Rb^{mN_p}}{\mbox{minimize}} & {\bf u}^\top H_1 {\bf u}^\top  + h^\top {\bf u} + z_0^\top H_2 {\bf u}  \\
\mbox{subject to} &  L {\bf u} + M z_0 \le d \\
\mbox{parameter} &  z_0 = {\bs {\hat \phi}}(x_{\mr{current}})
\end{array}
\end{equation}
for some matrices $H_1$, $H_2$, $L$, $M$ and vectors $h$, $d$ (explicit expressions in terms of the data of~(\ref{eq:MPC}) are in the Appendix). Notice that once the product $z_0^\top H_2$ is evaluated, the cost of solving the optimization problem~(\ref{eq:MPC_dense}) is \emph{independent} of the number of eigenfunctions $N$ used. This is essential for practical applications since $N$ can be large in order to ensure a small prediction error~(\ref{eq:predError}).  The optimization problem~(\ref{eq:MPC_dense}) is a convex QP that can be solved by any of the generic  packages for convex optimization (e.g., MOSEK or Gurobi) but also using highly tailored tools exploiting the specifics of the MPC formulation. In this work, we relied on the qpOASES package~\cite{ferreau2014qpoases} that uses a homotopy-based active set method which is particularly suitable for dense-form MPC problems and effectively utilizes warm starting to reduce the closed-loop computation time.

The closed-loop operation of the Koopman MPC is summarized in Algorithm~\ref{alg:koopmanMPC}. Here we assume sampled-data operation, where the control input is computed every $T_{\mr{s}}$ seconds and held constant between the sampling times. We note, however, that the mapping
\[
x_{\mr{current}} \mapsto {\bf u}_0^\star,
\]
where ${\bf u}_0^\star$ is the first component of the optimal solution ${\bf u}^\star = [{\bf u}_0^\star,\ldots,{\bf u}_{N_{\mr{p}}-1}^\star]^\top$ to the problem~(\ref{eq:MPC_dense}), defines a feedback controller that can be evaluated at an arbitrary state $x \in \Rb^n$ at an arbitrary time.

\begin{algorithm}
\caption{Koopman MPC -- closed-loop operation}\label{alg:koopmanMPC}
\begin{algorithmic}[1]
\For{$k=0,1,\ldots$}
\State Set $x_{\mr{current}} = x(kT_{\mr{s}})$ (current state of~(\ref{eq:sys_cont}))
\State Compute $z_0 = \bs {\bs {\hat \phi}}(x_{\mr{current}} )$
\State Solve~(\ref{eq:MPC_dense}) to get an optimal solution ${ \bf u }^\star = [{\bf u}_0^\star,\ldots,{\bf u}_{N_{\mr{p}}-1}^\star]^\top$
\State Apply ${\bf u}_0^\star$ to the system~(\ref{eq:sys_cont}) for $t\in \big[kT_{\mr{s}}\,,\,(k+1)T_{\mr{s}}\big)$
\EndFor
\end{algorithmic}
\end{algorithm}

\section{Numerical examples}

In the numerical examples we investigate the performance of the predictors on the Van der Pol oscillator and the damped Duffing oscillator. The two dynamical systems exhibit a very different behavior: The former has a stable limit cycle whereas the latter two stable equilibria and an unstable equilibrium. However, interestingly but in line with the theory, we observe a very good performance of the predictors constructed for both systems, away from the limit cycle and singularities. On the Duffing system, we also investigate feedback control using the Koopman MPC, managing both transition between the two stable equilibria as well stabilization of the unstable one, in a purely data-driven and convex-optimization-based fashion. Matlab code for the numerical examples is available from

\begin{center}
\url{https://homepages.laas.fr/mkorda/Eigfuns.zip}
\end{center}

\subsection{Van der Pol oscillator}
\begin{figure*}[th]
\begin{picture}(140,155)

\put(30,3){\includegraphics[width=70mm]{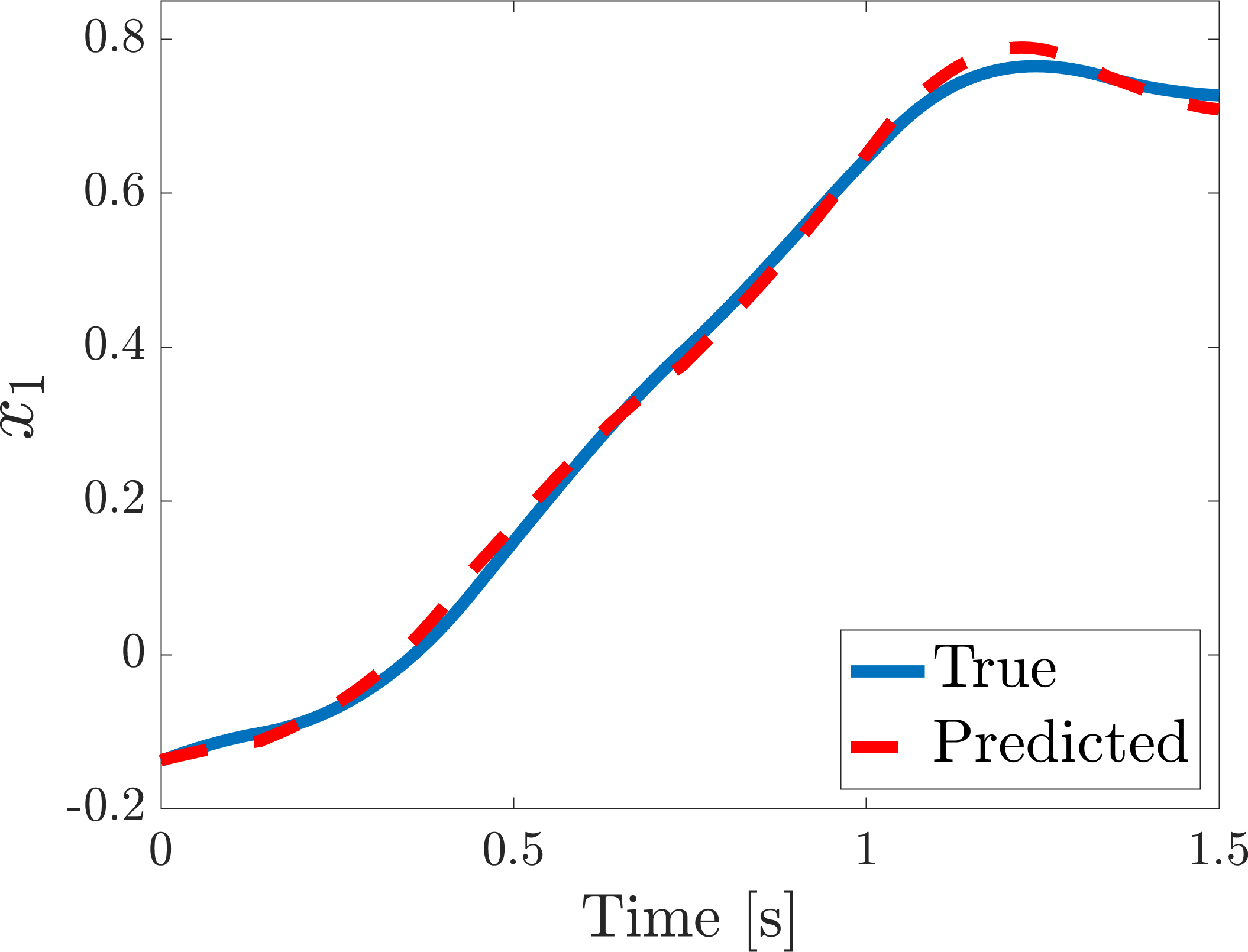}}
\put(250,3){\includegraphics[width=70mm]{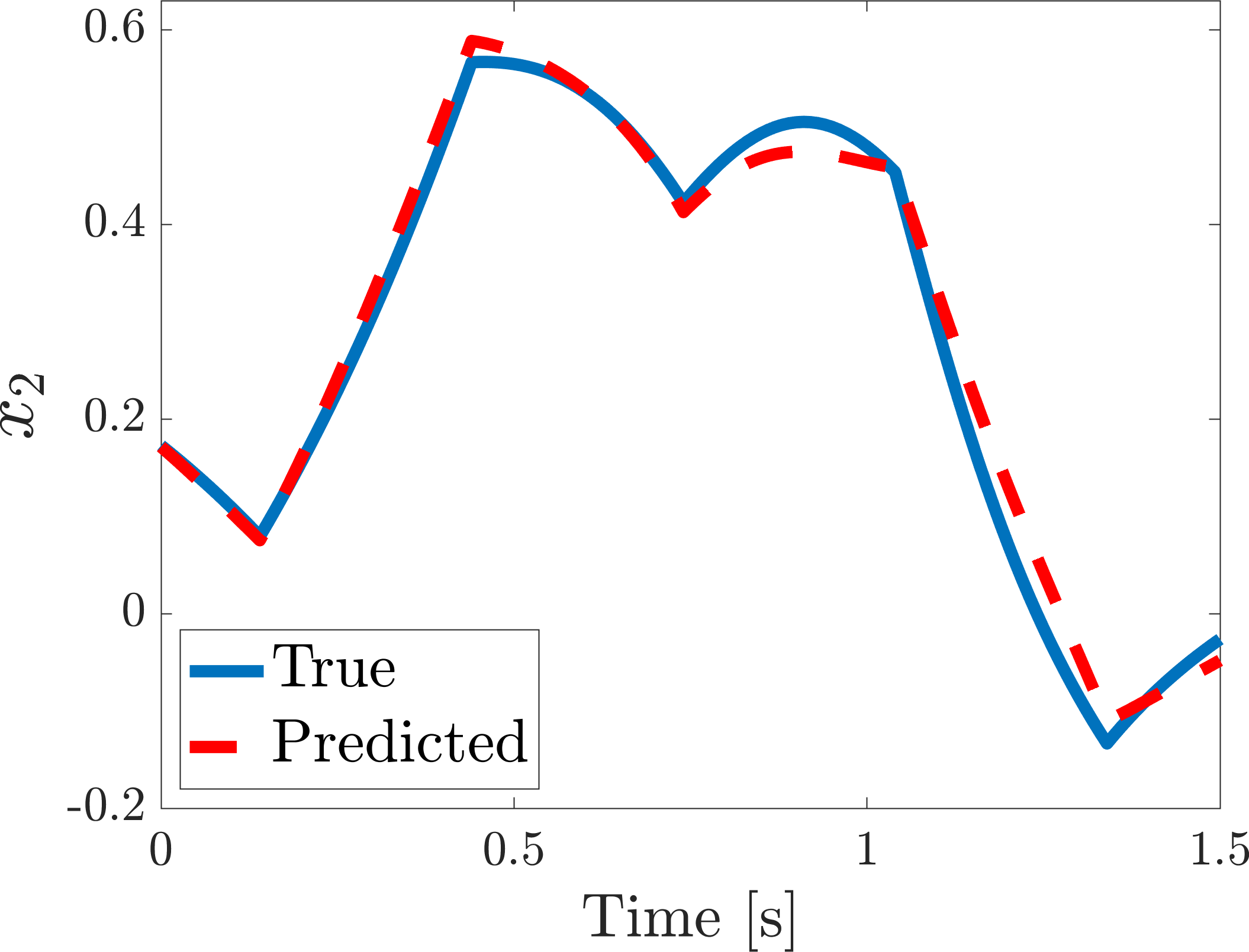}}

\end{picture}
\caption{\footnotesize \new{Van der Pol oscillator -- Prediction with 20 eigenfunctions with optimized selection of eigenvalues and boundnary functions, a randomly chosen initial condition and square wave forcing.}}
\label{fig:vp_pred}
\end{figure*}

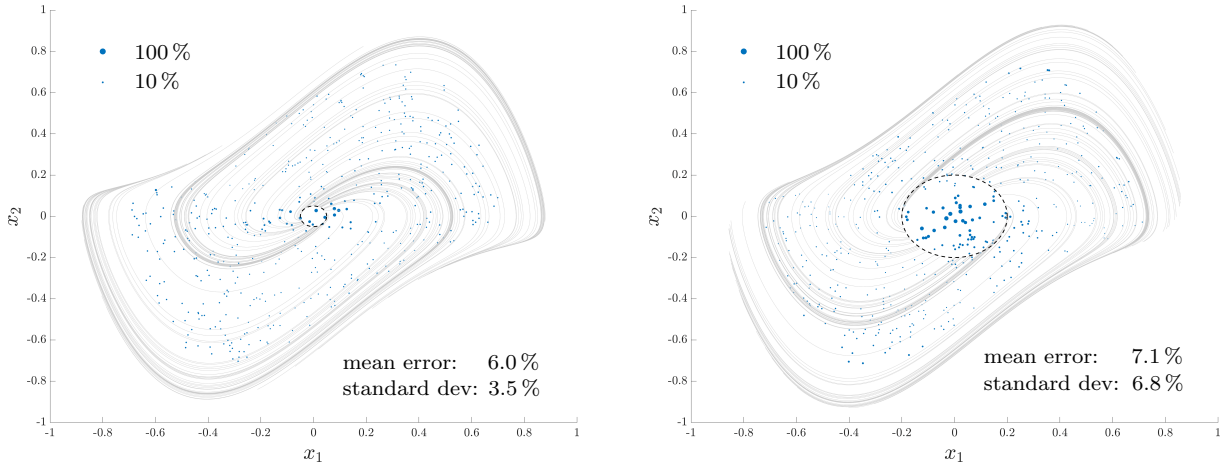
\begin{figure*}[th]
\begin{picture}(140,185)

% Compressed using https://compresspng.com/
\put(0,3){\includegraphics[width=75mm]{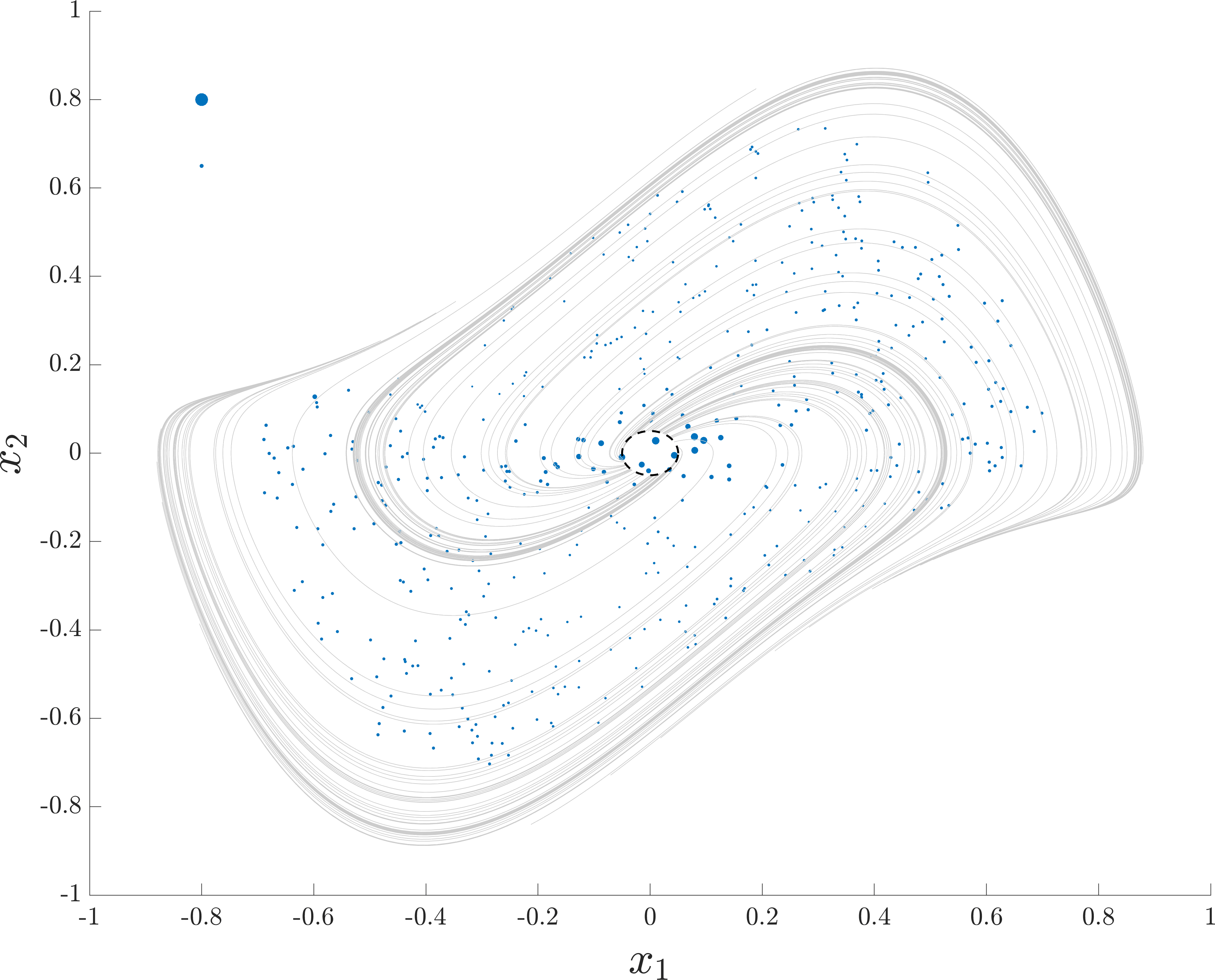}}
\put(240,3){\includegraphics[width=75mm]{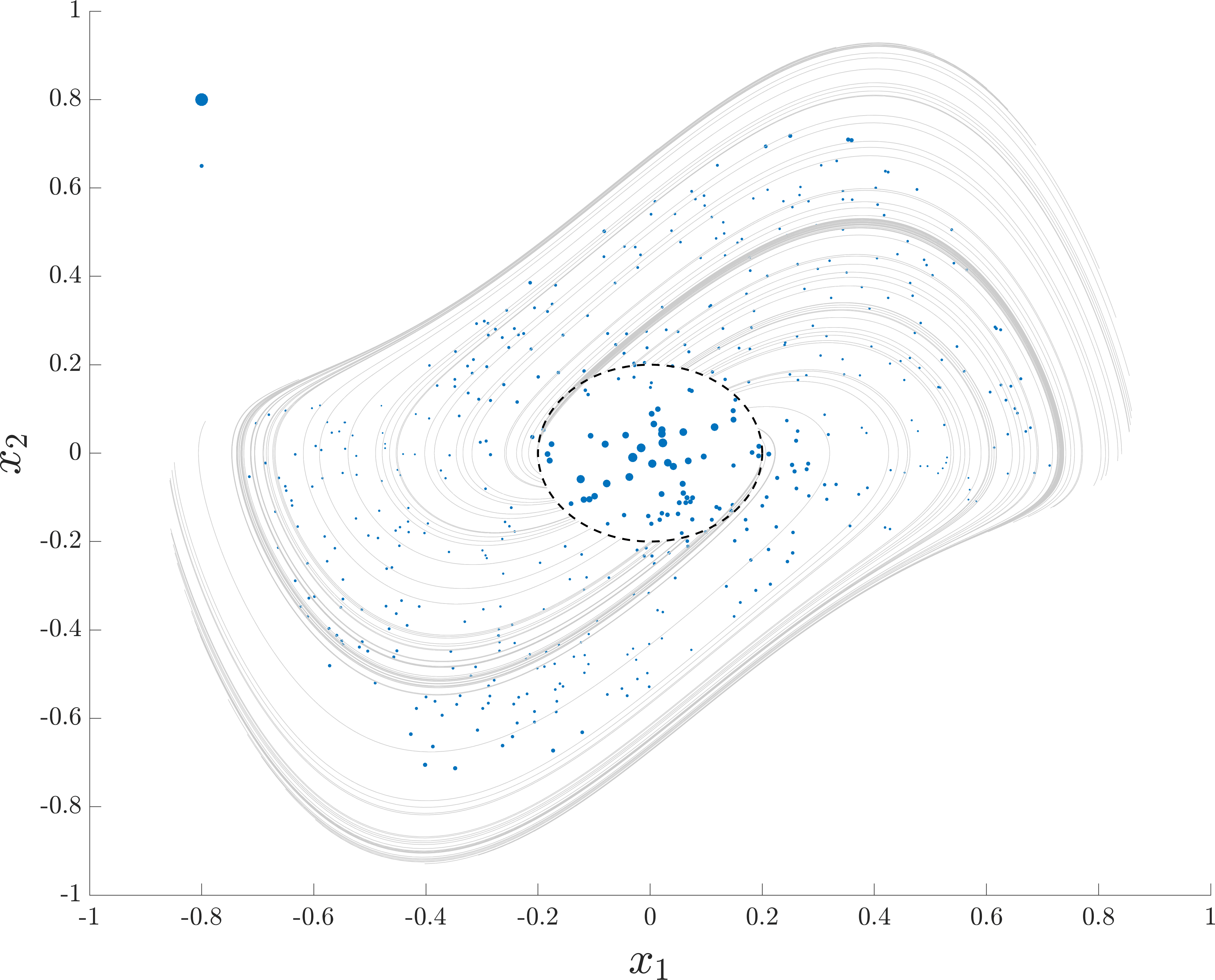}}

\put(47,155){\scriptsize $ 100 \, \%$}
\put(47,144){\scriptsize $ 10 \, \%$}
\put(287,155){\scriptsize $ 100 \, \%$}
\put(287,144){\scriptsize $ 10 \, \%$}

%\put(125,38){\scriptsize mean error: \;\, $ 7.1 \, \%$}
%\put(125,28){\scriptsize standard dev: $ 6.8 \, \%$}
%\put(365,40){\scriptsize mean error: \;\;\;\,$ 6.0 \, \%$}
%\put(365,30){\scriptsize standard dev: \hspace{-1mm} $ 3.5 \, \%$}

\put(125,38){\scriptsize mean error: \;\, $ 6.0 \, \%$}
\put(125,28){\scriptsize standard dev: $ 3.5 \, \%$}
\put(365,40){\scriptsize mean error: \;\;\;\,$ 7.1 \, \%$}
\put(365,30){\scriptsize standard dev: \hspace{-1mm} $ 6.8 \, \%$}

\end{picture}
\caption{\footnotesize \new{Van der Pol oscillator -- Spatial distribution of the prediction error (controlled) with 20 eigenfunctions with optimized selection of eigenvalues and boundary functions. The trajectories used for construction of the eigenfunctions are depicted in grey.  Their initial conditions were sampled from a circle of radius 0.05 (left pane) and 0.2 (right pane), both depicted in dashed black; neither circle is a non-recurrent surface for the dynamics (which is \emph{not} required by the method). The error for each of the 500 initial conditions from the independently generated test set is encoded by the size of the blue marker.}}
\label{fig:vpSpatialError}
\end{figure*}

In the first example, we consider the classical Van der Pol oscillator with forcing
\begin{align*}
    \dot{x}_1 &= 2x_2 \\
    \dot{x}_2 &= -0.8x_1 + 2x_2 - 10x_1^2x_2 + u.
\end{align*}

We investigate the performance of the proposed predictors, both in controlled and uncontrolled (i.e., $u = 0$) settings. \new{The function $\ob$ to predict  is the state itself, i.e. $\ob(x) = x$, and hence the output $\hat y$ of (\ref{eq:linPred_uncont}) and (\ref{eq:linPred_cont}) predicts the state of the system.  We investigate the performance of the proposed predictor as a function of the number of eigenfunctions $N$ used.   First, we construct the eigenfunction approximations as described in Section~\ref{sec:learning}. The budget of $N$ eigenfunctions is split equally among the components of $\ob$, i.e., $N_1=N_2 = N/2$.} We generate a set of $M_{\mr{t}} = 100$ five second long trajectories sampled with a sampling period $T_{\mr{s}} = 0.01\,\mr{s}$ (i.e.,  $M_{\mr{s}} = 500$). The initial conditions of the trajectories are sampled uniformly over a circle of radius 0.05. As a first and natural choice of eigenvalues $\Lambda$, we utilize $\Lambda_{\mr{lat}} = \lattice_{d_\mr{lat}}( \frac{1}{T_{\mr{s}}} \log \Lambda_\mr{DMD})$, where $\Lambda_\mr{DMD}$ are the two eigenvalues obtained by applying the dynamic mode decomposition algorithm to the data set and
\begin{equation}\label{eq:meshd}
\lattice\nolimits_d(\Lambda) =  \Big\{ \sum_{k=1}^{p} \alpha_k \lambda_k \mid \lambda_k \in \Lambda,\; \alpha_k\in \Nb_0, \; p\in \Nb, \; \sum_{k=1}^p \alpha_k \le d \Big\}.
\end{equation}
\new{The number of eigenvalues obtained in this way is $N_{\mr{lat}}= \binom{2 + d_\mr{lat}}{d_\mr{lat}}$. For each value of $N$, we choose $d_{\mr{lat}}$ such that $N_{\mr{lat}} \ge N_1=N_2=N/2$ and use the first $N/2$ eigenvalues of $\Lambda_{\mr{lat}}$ in the algorithm of Section~\ref{sec:optimChoice_lam} for optimal choice of the boundary functions for each component of $\ob$; we do not use regularization, i.e., the matrix $\mathbf{G}$ is obtained using (\ref{eq:gopt}) and (\ref{eq:G_data_optim}). Second, we investigate the benefit of optimizing the eigenvalues as described in Section~\ref{sec:optimChoice_lam}; the objective function~(\ref{eq:p_simplified}) is minimized using local Newton-type algorithm implemented in Matlab's fmincon, with analytic gradients computed using~(\ref{eq:p_grad}) and initial condition given by $\Lambda_{\mr{lat}}$.
}

The $N$ eigenfunctions are computed on the data set using~(\ref{eq:eigFun_computed}) and linear interpolation is used to define them on the entire state space. The $C$ matrix is computed using~(\ref{opt:projL2}) with $\bar x_i$ being the data used to construct the eigenfunctions plus a random noise uniformly distributed over $[-0.05,0.05]^2$. This fully defines the linear predictor~(\ref{eq:linPred_uncont}) in the uncontrolled setting. To get the $B$ matrix in the controlled setting we generate a second data set with forcing. The initial conditions are the same as in the uncontrolled setting; the forcing is piecewise constant signal taking a random uniformly distributed value in $[-1,1]$  in each sampling interval; the length of each trajectory is two seconds. The matrix $B$ and its discrete counterpart $B_{\mr{d}}$ are then obtained using~(\ref{eq:Bd}) and (\ref{eq:B}). \new{In the controlled setting, we investigate the prediction performance for two control signals distinct from the signal used during identification. The first one is a square wave with unit amplitude and period $300\,\mr{ms}$ and the second one a sinusoid wave with unit amplitude and period $60\,\mr{ms}$. Figure~\ref{fig:vp_pred} shows the true and predicted trajectories for the randomly chosen initial condition  $x_0 = [-0.1382,0.1728]^\top$. Table~\ref{tab:rmse_vp_N} reports the prediction error over one second time interval as a function of $N$. }  The error is reported as the root mean square error 
\begin{equation}\label{eq:rmse}
\mr{Prediction\; error} = 100\cdot \frac{\displaystyle\sqrt{\sum_{k}\|x_{\mr{pred}}(kT_s) - x_{\mr{true}}(k T_s)\|_2^2}}{\displaystyle\sqrt{ \sum_{k}\|x_{\mr{true}}(k T_s)\|_2^2}}.
\end{equation}
averaged over 500 randomly chosen initial conditions in the interior of the limit cycle. \new{We observe that optimization of the eigenvalues brings about a significant improvement in performance, especially with a small number of eigenfunctions. We also observe that without control, the average prediciton error is close to $1\,\%$ with only 20 eigenfunctions.}

\begin{table}[t]
\centering
\caption{\footnotesize \rm \new{Van der Pol oscilator -- Prediction error averaged over 500 randomly chosen initial conditions as a function of the total number of eigenfunctions $N$, with and without optimization of the eigenvalues $\lambda$.} }\label{tab:rmse_vp_N}\vspace{2mm}
{\small
\begin{tabular}{lccccc}
\toprule
\new{$\#$ of eigenfunctions $N$}                         &           \new{4}       &       \new{8}             &       \new{12}                    &           \new{16}    & 		 \new{20}                            \\\midrule
\new{$\lambda$ \textbf{not optimized}} & & & & & \\
  \quad \new{Prediction error [uncontrolled]}  &   \new{$ 100.4\,\%$}   &  \new{$95.6 \,\%$}    &          \new{ $51.34\,\%$}       &       \new{$ 13.31 \,\%$}     &  \new{ $5.44\,\%$}     \\
    \quad \new{Prediction error  [square wave control]}  &   \new{$102.2\,\%$}   &  \new{$115.7 \,\%$}    &           \new{$51.2\,\%$}       &       \new{$ 14.5  \,\%$}     &   \new{$8.3\,\%$}    \\
        \quad \new{Prediction error  [sinus wave control]}  &   \new{$ 101.3\,\%$}   &  \new{$97.0 \,\%$}    &           \new{$53.3\,\%$}       &       \new{$ 13.8 \,\%$}     &   \new{$6.2\,\%$}    \\\midrule
        \new{$\lambda$ \textbf{optimized}} & & & & & \\
  \quad \new{Prediction error  [uncontrolled]}  &   \new{$ 24.0\,\%$}   &  \new{$9.7 \,\%$}    &           \new{$5.1\,\%$}       &       \new{$ 2.4  \,\%$}     &   \new{$1.4\,\%$}     \\
    \quad \new{Prediction error  [square wave control]}  &   \new{$ 27.6\,\%$}   &  \new{$11.0 \,\%$}    &           \new{$7.8\,\%$}       &       \new{$ 6.7  \,\%$}     &   \new{$6.0\,\%$}    \\
        \quad \new{Prediction error  [sinus wave control]}  &  \new{ $ 25.5\,\%$}   & \new{ $10.3 \,\%$}    &           \new{$5.8\,\%$}       &       \new{$ 3.7  \,\%$}     &   \new{$2.9\,\%$}\\       
\bottomrule
\end{tabular}
}
\end{table}

\new{ Next, we investigate the spatial distribution of the prediction error as a function of the initial condition. We report results for two sets of data -- the original set as described above and a second set with initial conditions starting  on a larger circle of radius 0.2 centered around the origin and trajectory length of three seconds. Figure~\ref{fig:vpSpatialError} reports the results (for brevity we depict only the results with square wave control, the other case being qualitatively similar).  We observe that choosing a smaller smaller circle to sample the initial conditions from results in a larger portion of the state-space being covered by the generated trajectories and hence smaller mean prediction error as well as its standard deviation.}

\subsection{Damped Duffing oscillator}
\begin{figure*}[th]
\begin{picture}(140,165)

%\put(30,3){\includegraphics[width=70mm]{./Figures/vp_pred_x1_00506_02957.pdf}}
%\put(250,3){\includegraphics[width=70mm]{./Figures/vp_pred_x2_00506_02957.pdf}}
\put(30,3){\includegraphics[width=70mm]{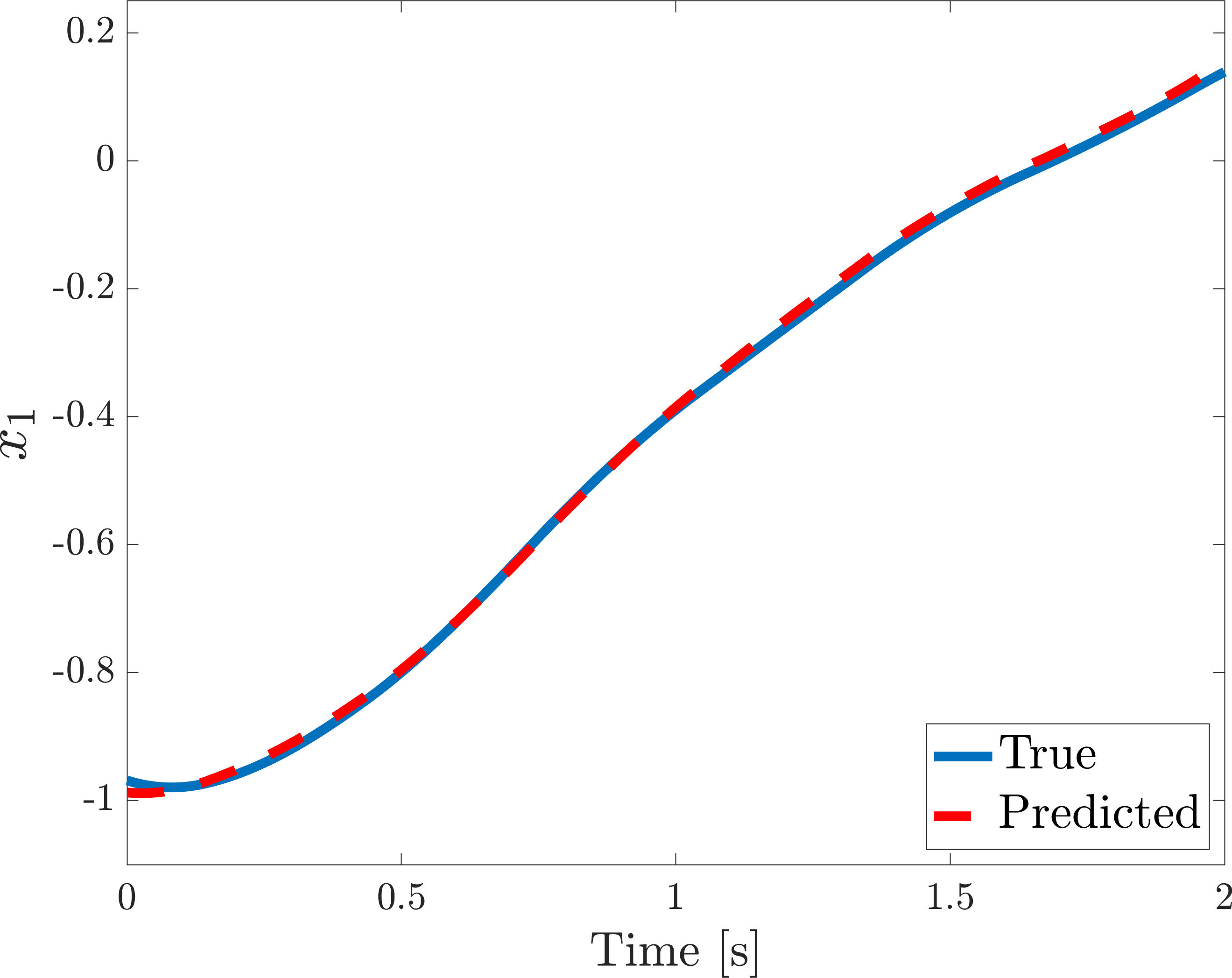}}
\put(250,3){\includegraphics[width=70mm]{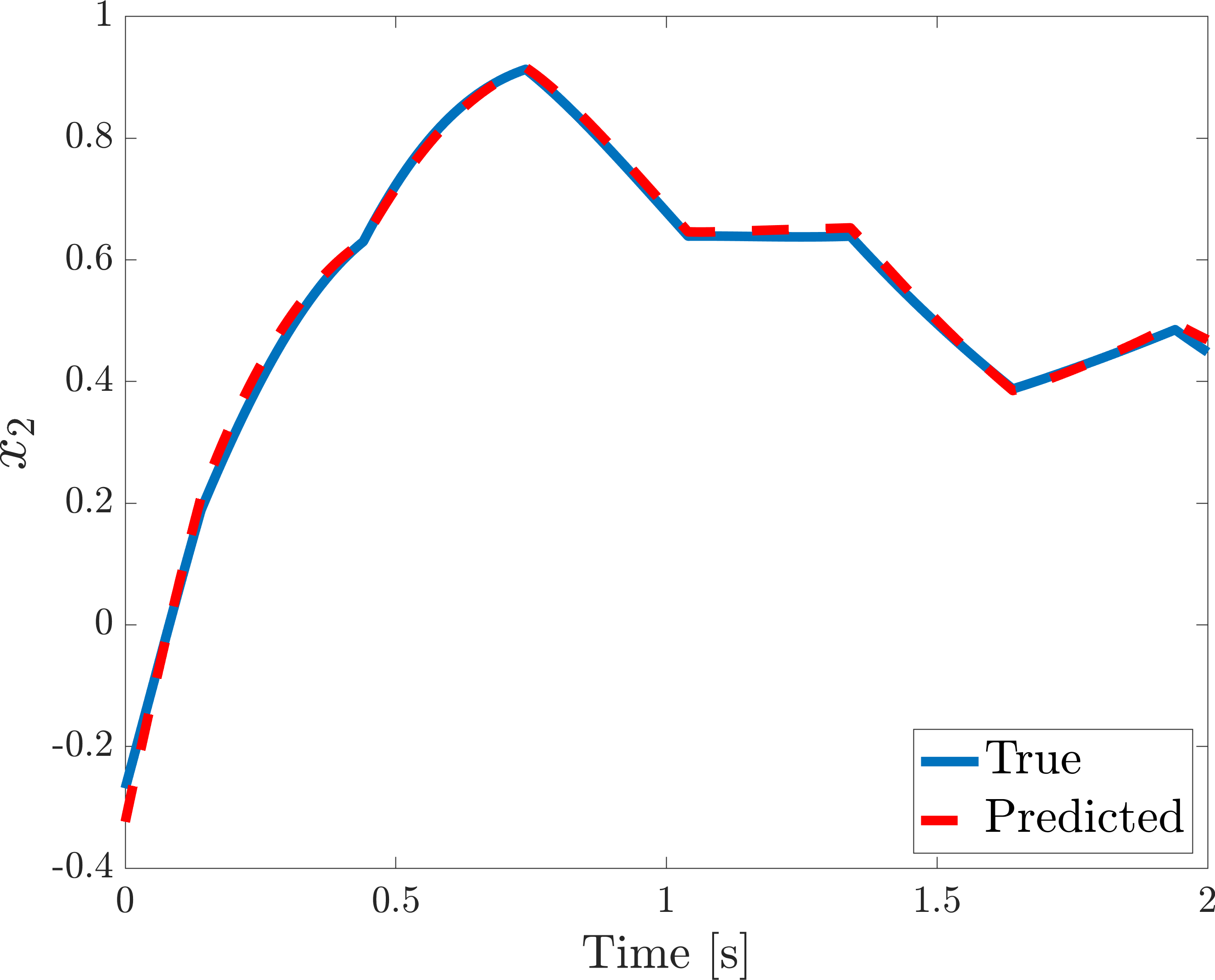}}

\end{picture}
\caption{\footnotesize \new{Duffing oscillator -- Prediction with twenty eigenfunctions with optimized selection of eigenvalues and boundary functions, a randomly chosen initial condition and square wave forcing.}}
\label{fig:duff_pred}
\end{figure*}

\begin{figure*}[th]
\begin{picture}(140,200)

%\put(-10,3){\includegraphics[width=90mm]{./Figures/duff_cont_Ntraj100_N201.pdf}}
%\put(230,3){\includegraphics[width=90mm]{./Figures/duff_cont_Ntraj50_N201.pdf}}

% Compressed using https://compresspng.com/
\put(-10,3){\includegraphics[width=81mm]{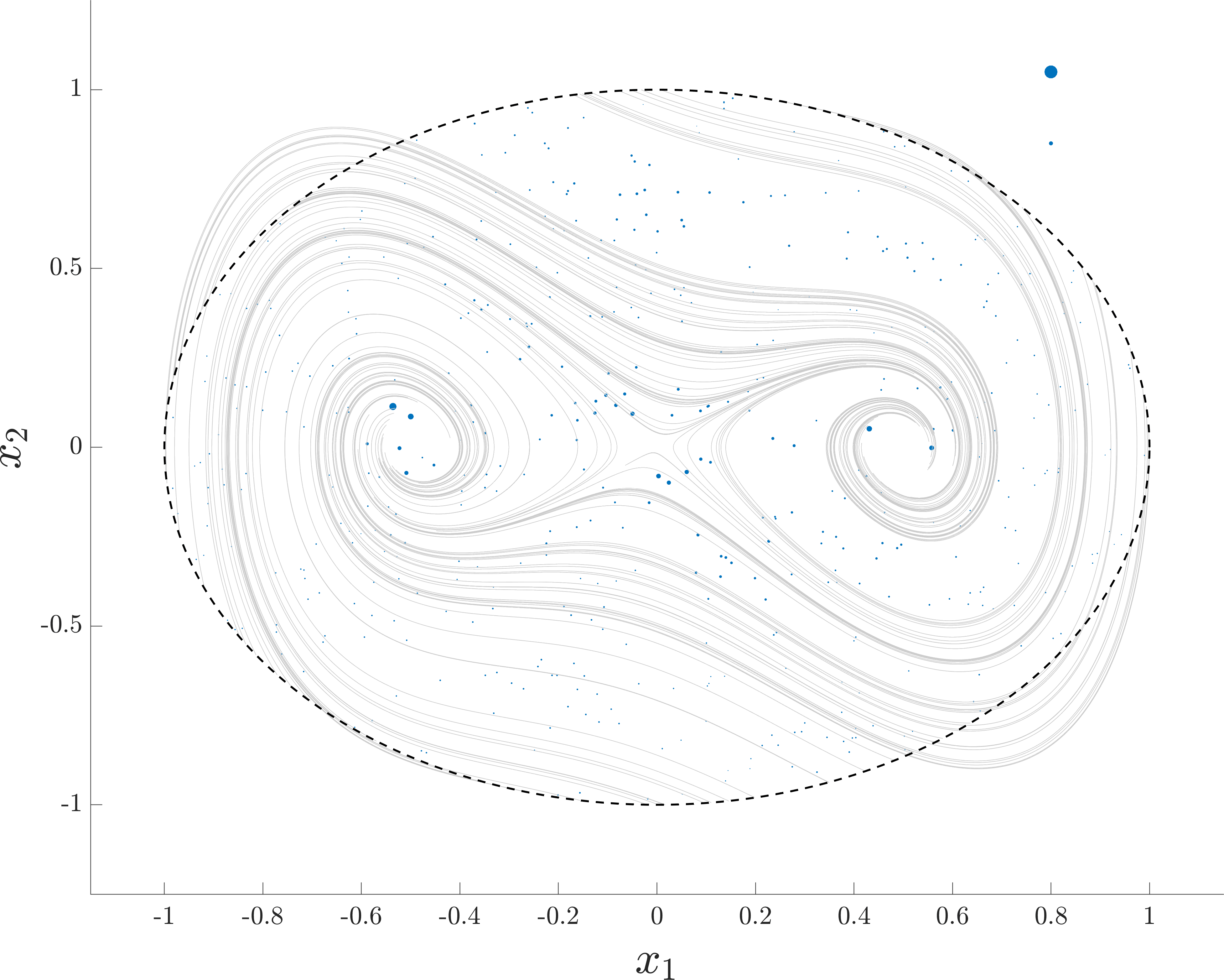}}
\put(230,3){\includegraphics[width=81mm]{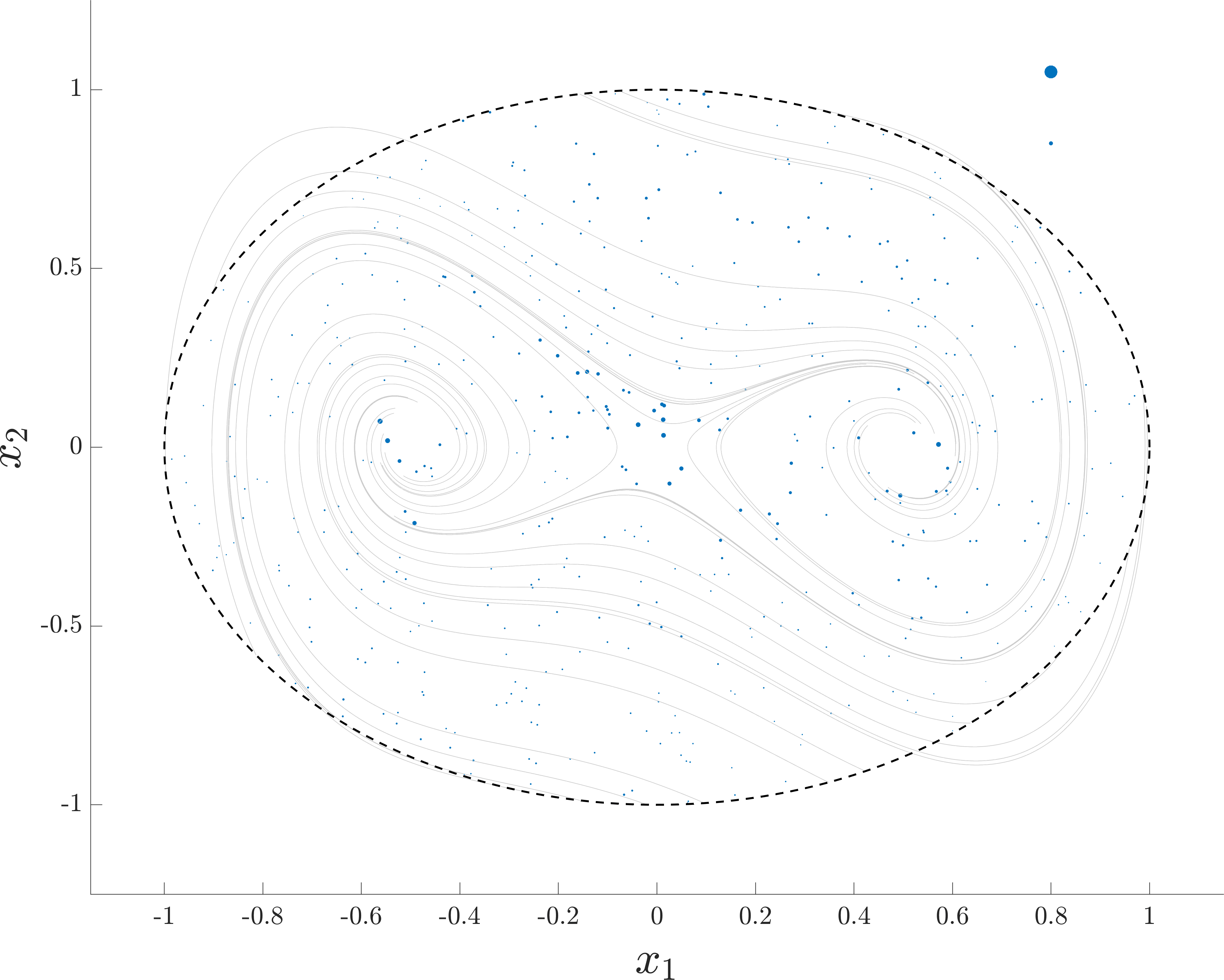}}

\put(197,170){\scriptsize $ 100 \, \%$}
\put(197,158){\scriptsize $ 10 \, \%$}
\put(437,170){\scriptsize $ 100 \, \%$}
\put(437,158){\scriptsize $ 10 \, \%$}

\put(87,180){\scriptsize $\#$\,Trajectories $= 100$}
\put(327,180){\scriptsize  \new{$\#$\,Trajectories $=25$}}

\put(155,34){\scriptsize \new{mean error: \;\, $ 2.1 \, \%$}}
\put(155,24){\scriptsize \new{standard dev: $ 1.9 \, \%$}}
\put(393,36){\scriptsize \new{mean error: \;\;\,$ 2.8 \, \%$}}
\put(393,26){\scriptsize \new{standard dev: \hspace{-1.8mm} $ 2.1 \, \%$}}

%\put(125,38){\scriptsize mean error: \;\, $ 7.7 \, \%$}
%\put(125,28){\scriptsize standard dev: $ 8.6 \, \%$}
%\put(365,40){\scriptsize mean error: \;\;\;$ 12.6 \, \%$}
%\put(365,30){\scriptsize standard dev: \hspace{0.1mm} $ 4.6 \, \%$}
%\put(325,0){\small $\mr{time} [s]$}
%\put(8,80){ $y$}
%\put(228,80){ $y$}
\end{picture}
\caption{\footnotesize\new{Damped Duffing oscillator -- Spatial distribution of the prediction error (square wave control signal) with twenty eigenfunctions with optimized selection of eigenvalues and boundary functions. The trajectories used for construction of the eigenfunctions are depicted in grey; their initial conditions were sampled uniformly from the unit circle (dashed black) -- note that the unit circle is \emph{not} a non-recurrent surface for the dynamics and this property is \emph{not} required by the method. The prediction error for each of the 500 initial conditions from the independently generated test set is encoded by the size of the blue marker. }}
\label{fig:duffSpatialError}
\end{figure*}
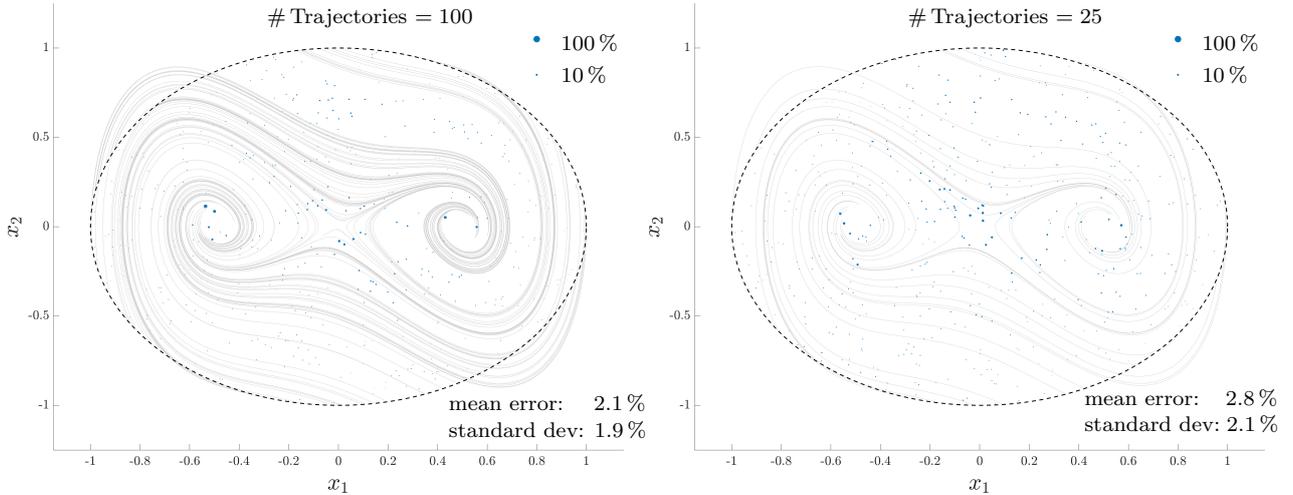

As our second example we consider the damped duffing oscillator with forcing
\begin{align*}
    \dot{x}_1 &= x_2 \\
    \dot{x}_2 &=  - 0.5x_2 - x_1(4x_1^2 - 1) + 0.5 u.
\end{align*}
The uncontrolled system (with $u = 0$) has two stable equilibria at $(-0.5,0)$ and $(0.5,0)$ as well as an unstable equilibrium at the origin.

\paragraph{Prediction} First we investigate the performance of the proposed predictors on this system. The setup is very similar to the previous example.  \new{The function $\ob$ to predict  is the state itself, i.e., $\ob(x) = x$, and we investigate the predictive performance as function of the number of eigenfunctions $N$ used.   First, we construct the eigenfunction approximations as described in Section~\ref{sec:learning}. The budget of $N$ eigenfunctions is split equally among the components of $\ob$, i.e., $N_1=N_2 = N/2$.} We generate $M_{\mr{t}} = 100$ eight second long trajectories sampled with a sampling period $T_{\mr{s}} = 0.01$ (i.e., $M_{\mr{s}} = 800$). The initial conditions are chosen randomly from a uniform distribution on the unit circle. We note that the unit circle is \emph{not} a non-recurrent surface for this system. 

\new{As before, for non-optimized eigenvalues we utilize $\Lambda_{\mr{lat}} = \lattice_{d_\lambda}( \frac{1}{T_{\mr{s}}} \log \Lambda_\mr{DMD})$, where $\Lambda_\mr{DMD}$ are the eigenvalues obtained from applying the DMD algorithm to the generated data set and $\lattice_{d_\lambda}(\cdot)$ is defined in~(\ref{eq:meshd}). For each value of $N$, we choose $d_{\mr{lat}}$ such that $N_{\mr{lat}}(d_\mr{lat}) \ge N_1=N_2=N/2$ and use the first $N/2$ eigenvalues of $\Lambda_{\mr{lat}}$ in the algorithm of Section~\ref{sec:optimChoice_lam} for optimal choice of the boundary functions for each component of $\ob$; we do not use regularization, i.e., the matrix $\mathbf{G}$ is obtained using (\ref{eq:gopt}) and (\ref{eq:G_data_optim}). Second, we investigate the benefit of optimizing the eigenvalues as described in Section~\ref{sec:optimChoice_lam}; the objective function~(\ref{eq:p_simplified}) is minimized using local Newton-type algorithm implemented in Matlab's fmincon, with analytic gradients computed using~(\ref{eq:p_grad}) and initial condition given by $\Lambda_{\mr{lat}}$.}

The $N $ eigenfunctions are computed on the data set using~(\ref{eq:eigFun_computed}) and linear interpolation is used to define them on the entire state space. The matrix $C$  is obtained using~(\ref{eq:Cbdiag}). To get the matrix $B$ in the controlled setting we generate data with forcing. The initial conditions are the same as in the uncontrolled setting; the forcing is piecewise constant signal taking a random uniformly distributed values in $[-1,1]$ in each sampling interval; the length of each trajectory is two seconds. The matrix $B$ and its discrete counterpart $B_{\mr{d}}$ are then obtained using~(\ref{eq:Bd}) and (\ref{eq:B}). \new{We investigate performance for two control signals distinct from the one used during identification. The first one is a square wave with unit amplitude and period $300\,\mr{ms}$ and the second one a sinusoid wave with unit amplitude and period $60\,\mr{ms}$.} Figure~\ref{fig:duff_pred} shows a prediction for a randomly chosen initial condition in the controlled setting with the the square wave forcing.

\new{Table~\ref{tab:rmse_duff_N} reports the prediction error~(\ref{eq:rmse}) averaged over 500 initial conditions chosen randomly inside the unit circle for different values of $N$. We observe that optimization of eigenvalues brings about a significant improvement, especially with a small number of eigenfunctions. We also notice that with 28 eigenfunctions, the prediction error is around one percent in both the controlled and uncontrolled scenarios. Figure~\ref{fig:duffSpatialError} then shows the spatial distribution of the prediction error over a one second prediction time interval, where we compare the predictors constructed from 100 trajectories and 25 trajectories. We observe that the prediction error deteriorates only very little, showing a remarkable robustness to the amount of data used (we tested this further and even with 10 trajectories, the mean error increased to only $5.4\,\%$). We also observed the locations of the optimized eigenvalues to be very robust with respect to the number of trajectories used.
}

%\begin{table}[t]
%\centering
%\caption{\footnotesize \rm Damped duffing oscillator -- Prediction error. The total number of eigenfunctions for each combination $(N_\Lambda, N_G) $ is $N  = N_\Lambda\cdot N_G$.}\label{tab:rmse_duff_N}\vspace{2mm}
%{\small
%\begin{tabular}{ccccccc}
%\toprule
%$(N_\Lambda, N_G) $           &           ($10,30$)     &            ($10,20$)        &         ($6,20$)              &         ($10,10$)                    &            ($10,5$)    & 		 ($10,3$)                            \\\midrule
%  Mean RMSE [uncontrolled]  &   $ 6.9\,\%$          &         $  8.9 \,\%$     &           $17.4\,\%$         &       $ 19.9  \,\%$                &   $38.8\,\%$       &    $56.2\,\%$                                                    \\
% Mean RMSE [controlled]        &   $ 4.6 \,\%$         &        $ 6.7 \,\%$     &           $15.8\,\%$       &       $ 15.7  \,\%$                 &   $35.6\,\%$        &  $53.5\,\%$                                                      \\
%\bottomrule
%\end{tabular}
%}
%\end{table}

\begin{table}[t]
\centering
\caption{\footnotesize \rm \new{Damped Duffing oscillator -- Prediction error averaged over 500 randomly chosen initial conditions as a function of the total number of eigenfunctions $N$, with and without optimization of the eigenvalues $\lambda$. }}\label{tab:rmse_duff_N}\vspace{2mm}
{\small
\begin{tabular}{lccccc}
\toprule
\new{$\#$ of eigenfunctions}                         &           \new{12}       &       \new{16}             &       \new{20}                    &           \new{24}   & 		 \new{28}         \\\midrule
\new{$\lambda$ \textbf{not optimized}} & & & & & \\
  \quad \new{Prediction error [uncontrolled]}  &   \new{$ 26.0\,\%$}   &  \new{$22.5 \,\%$}    &          \new{$8.3\,\%$}       &       \new{$ 2.7  \,\%$}     &   \new{$1.1\,\%$}     \\
    \quad \new{Prediction error  [square wave control]}  &   \new{$ 26.0\,\%$}   &  \new{$22.1 \,\%$}    &           \new{$7.6\,\%$}       &       \new{$ 2.7  \,\%$}     &   \new{$1.5\,\%$}    \\
        \quad \new{Prediction error  [sinus wave control]}  &   \new{$ 25.2\,\%$}   &  \new{$21.2 \,\%$}    &           \new{$7.3\,\%$}       &       \new{$ 2.4  \,\%$}     &   \new{$1.1\,\%$}    \\\midrule
        \new{$\lambda$ \textbf{optimized}} & & & & & \\
  \quad \new{Prediction error  [uncontrolled]}  &   \new{$ 11.2\,\%$}   &  \new{$4.7 \,\%$}    &           \new{$2.1\,\%$}       &       \new{$1.1  \,\%$}     &   \new{$0.8\,\%$}     \\
    \quad \new{Prediction error  [square wave control]}  &   \new{$ 10.6\,\%$}   &  \new{$5.0 \,\%$}    &           \new{$2.3\,\%$}       &       \new{$ 1.5  \,\%$}     &   \new{$1.3\,\%$}    \\
        \quad \new{Prediction error  [sinus wave control]}  &   \new{$ 10.4\,\%$}   &  \new{$4.4 \,\%$}    &           \new{$2.0\,\%$}       &       \new{$ 1.2  \,\%$}     &   \new{$0.9\,\%$}\\
        
\bottomrule
\end{tabular}
}
\end{table}

%-------------------------------------------------------------

\paragraph{Feedback control}
\begin{figure*}[t!]
\begin{picture}(140,320)
\put(50,170){\includegraphics[width=60mm]{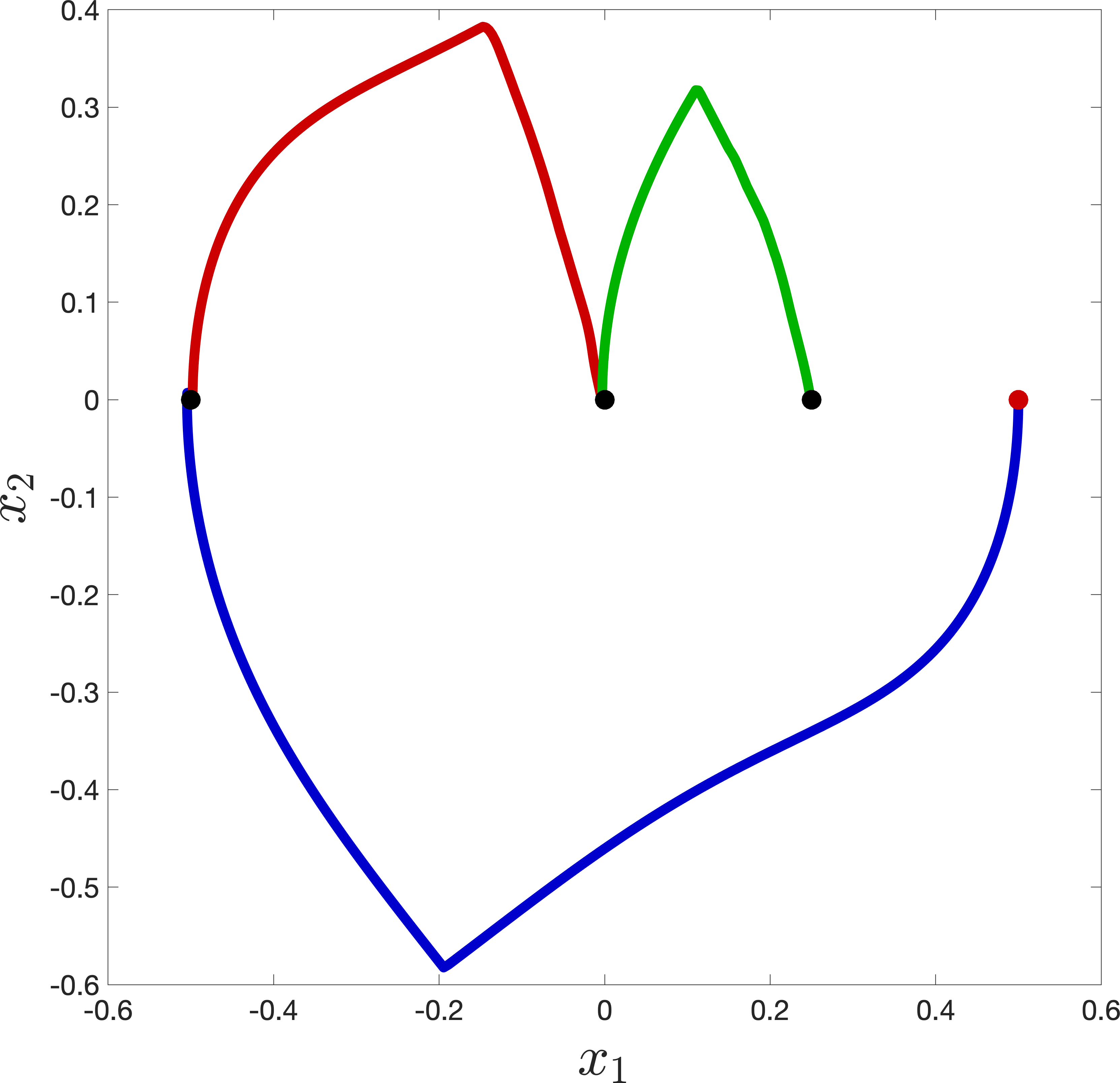}}
\put(250,170){\includegraphics[width=60mm]{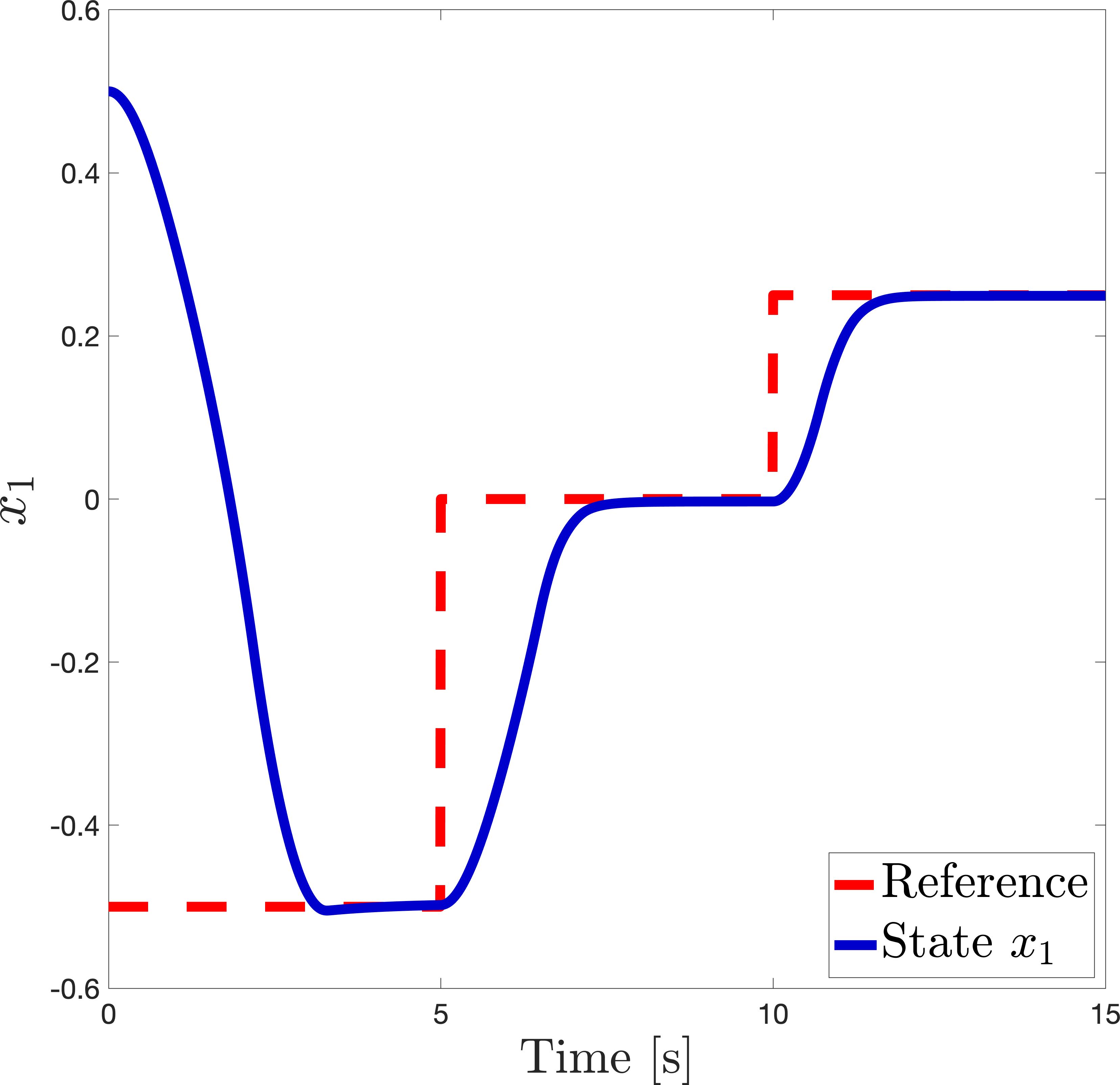}}
\put(50,0){\includegraphics[width=60mm]{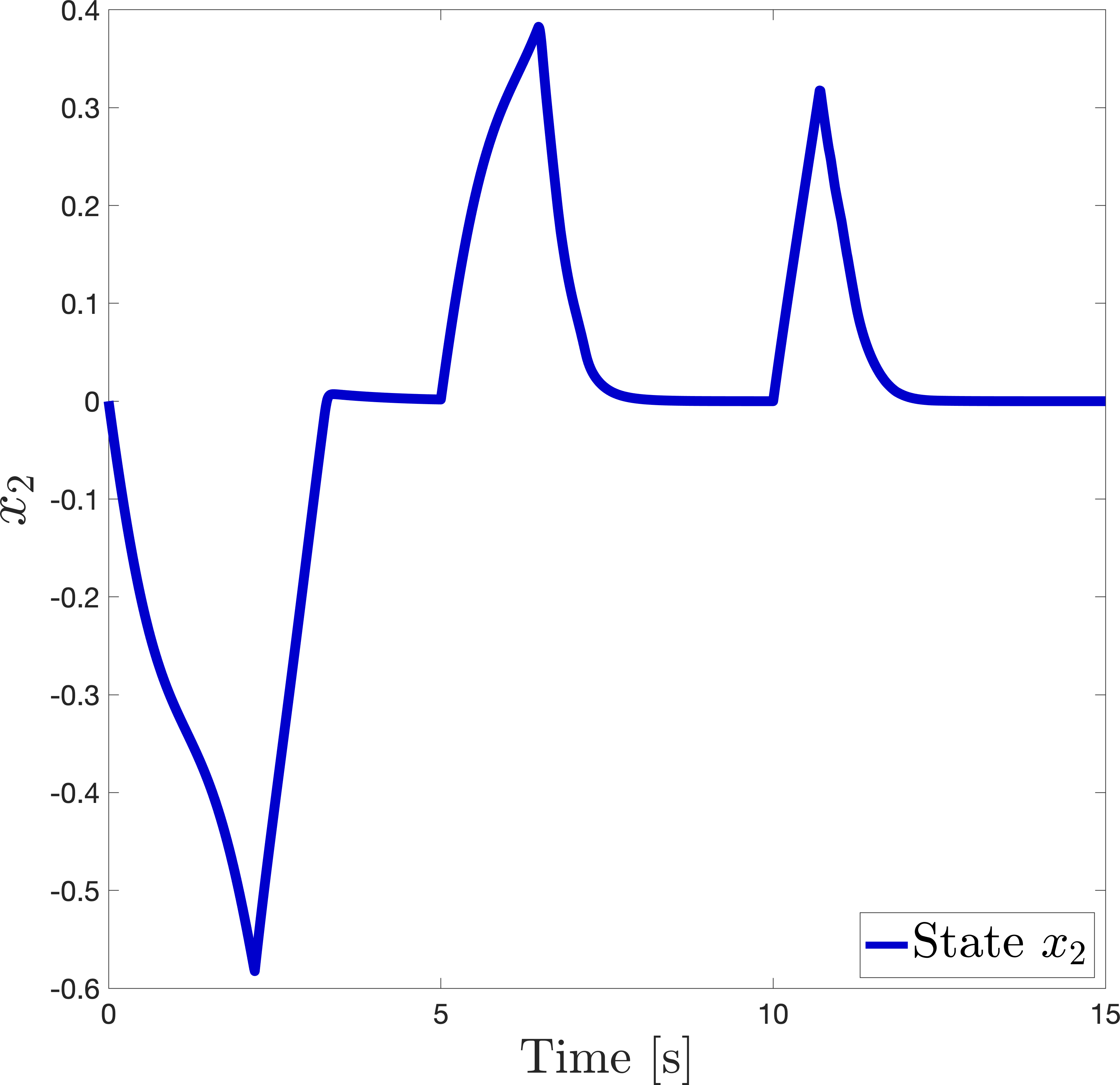}}
\put(250,0){\includegraphics[width=60mm]{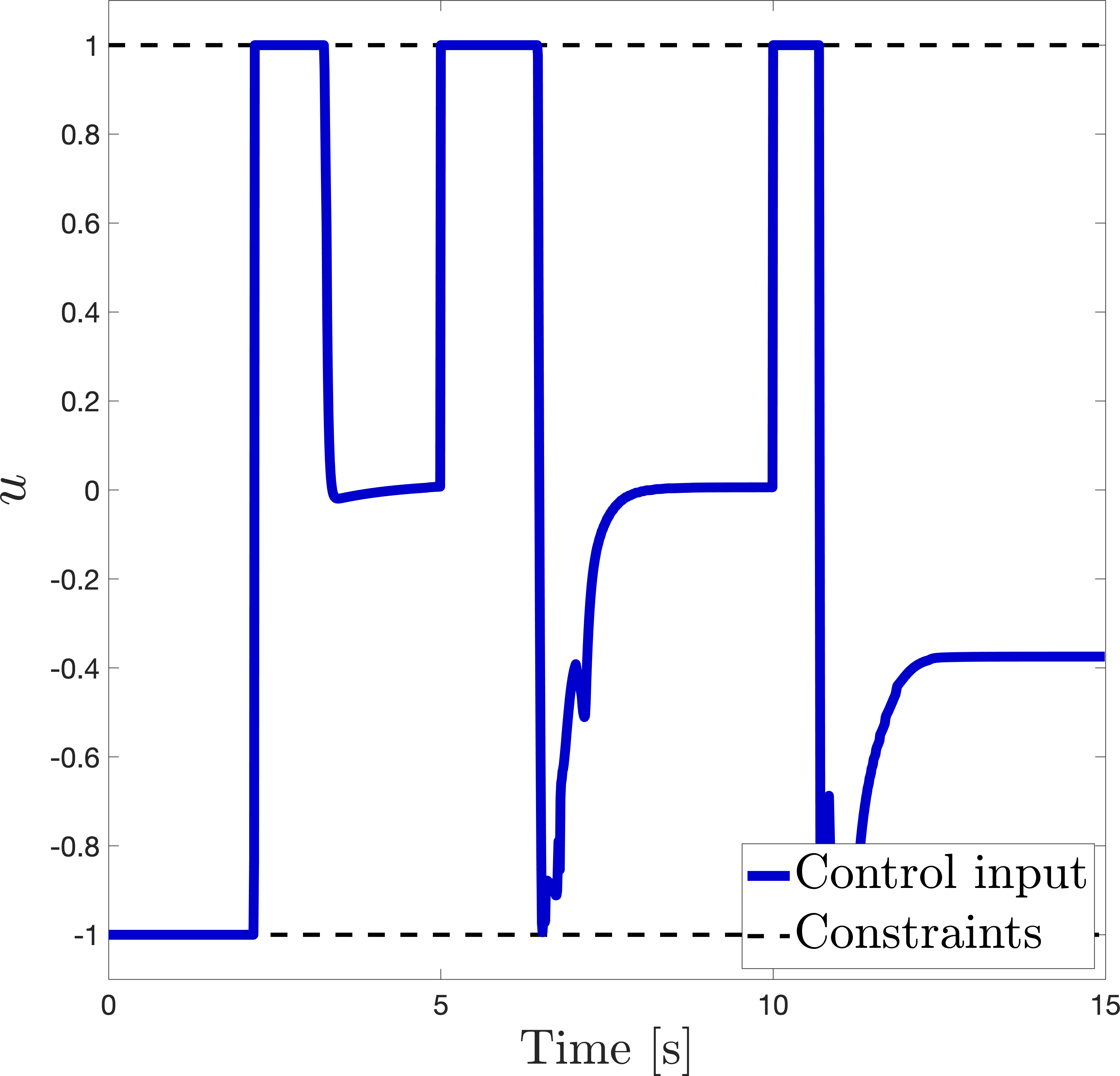}}

\put(120,250){\scriptsize Phase space}

\end{picture}
\caption{\small Duffing oscillator -- feedback control using Koopman MPC.}
\label{fig:duffingControl}
\end{figure*}

\new{Next, we apply the Koopman MPC developed in Section~\ref{sec:KoopmanMPC} to control the system with twenty eigenfunctions with optimimized selection of the eigenvalues and boundary functions. The goal is to track a piecewise constant reference signal, where we move from one stable equilibrium to the other $(0.5,0)\mapsto (-0.5,0)$, continue to the unstable saddle point at the origin and finish at $(0.25,0)$ which is not an equilibrium point for the uncontrolled system but is stabilizable for the controlled one. The matrices $Q$ and $R$ in~(\ref{eq:MPC}) were chosen $Q = \mr{diag}(1,0.1)$ and $R = 10^{-4}$ and we imposed the constraint $u\in [-1,1]$ on the control input. The prediction horizon was set to one second, i.e., $N_{\mr{p}} = 1.0 / T_{\mr{s}} = 100$.  The results are depicted in Figure~\ref{fig:duffingControl}; the tracking goal was achieved as desired. During the closed-loop operation, the MPC problem~(\ref{eq:MPC_dense}) was solved using the qpOASES solver~\cite{ferreau2014qpoases}.  The average computation time per time step was $0.59\,\mr{ms}$, of which approximately $0.43\,\mr{ms}$ was spent evaluating the eigenfunction mapping $\bs {\hat \phi}$ (Step 3 of~Algorithm~\ref{alg:koopmanMPC}) whereas the solution to the quadratic program~(\ref{eq:MPC_dense}) took on average only $0.16\,\mr{ms}$. The evaluation of $\bs {\hat \phi}$ could be significantly sped up by a more sophisticated interpolation implementation. We emphasize that the entire design was purely data driven and based only on linear model predictive control.} %, thereby allowing for a straightforward deployment in real-world applications.

\section{Conclusion}
This work presented a systematic framework for data-driven learning of Koopman eigenfunctions in non-recurrent regions of the state-space. The method is geared toward prediction and control using \emhp{linear predictors}, allowing for feedback control and state estimation for nonlinear dynamical systems using established tools for linear systems that rely solely on convex optimization or simple linear algebra. The proposed method exploits the richness of the spectrum of the Koopman operator away from attractors to construct a large number of eigenfunctions in order to minimize the projection error of the state (or any other observable of interest) on the span of the eigenfunctions. The proposed method is purely data-driven and very simple, relying only on linear algebra and/or convex optimization, with computation complexity comparable to DMD-type methods \new{and with very little user-input required; in particular only an interpolation method has to be selected, whereas the boundary functions and eigenvalues are determined from data using numerical optimization.}

Future work will explore possible extensions of the method to the case where recurrences are present.

\section*{Appendix}\label{ap:matrices}
 The matrices in~(\ref{eq:MPC_dense}) are given in terms of the data of~(\ref{eq:MPC}) by
\[
H_1 = {\mathbf{R}} +{ \mathbf{B}}^\top\mathbf Q {\mathbf{B}},\quad h = {\mathbf{q}}^\top \mathbf B + \mathbf {r}^\top,\quad H_2 = 2\mathbf A^\top\bf Q \mathbf {B},
\]
\[
L = \mathbf{F}+\mathbf{E}\mathbf{B},\quad M = \mathbf{E}\mathbf{A}, \quad d = [b_0^\top,\ldots,b_{N_\mr{p}}^\top]^\top,
\]
where
{
\[
\mathbf{A} = 
\begin{bmatrix}
I \\ CA_{\mr{d}} \\ CA_{\mr{d}}^2\\ \vdots\\ CA_{\mr{d}}^{N_{\mr{p}}}
\end{bmatrix}, \,
\mathbf{B} = \begin{bmatrix}0&0&\ldots&0\\
CB_{\mr{d}} &0&\ldots&0\\
CA_{\mr{d}}B_{\mr{d}} & CB_{\mr{d}} & \ldots  & 0 \\
 \vdots &\ddots&\ddots\\
 CA_{\mr{d}}^{N_p-1}B_{\mr{d}} & \ldots & CA_{\mr{d}}B_{\mr{d}} & CB_{\mr{d}} \end{bmatrix},\, \mathbf{F}=\begin{bmatrix}
 F_0& 0 &\ldots &0\\
 0 & F_1 & \ldots & 0\\
 \vdots & & \ddots & \vdots \\
 0 &  0 & \ldots & F_{N_\mr{p}-1}\\
 0 & 0 & \ldots & 0
 \end{bmatrix}
\]
\[
\mathbf{Q} = \mathrm{diag}(Q_0,\ldots,Q_{N_\mr{p}}),\;\; \mathbf{R} = \mathrm{diag}(R_0,\ldots,R_{N_\mr{p}-1}),
\]
\[
\mathbf{E} = \mathrm{diag}(E_0,\ldots,E_{N_\mr{p}}),\,\mathbf{q} = [q_0,\ldots, q_{N_\mr{p}}],\;\; \mathbf{r} = [r_0,\ldots,r_{N_\mr{p}-1}].
\]
}

\section{Acknowledgments}
\new{The authors would like to thank Corbinian Schlosser for pointing out to us a new proof of Theorem~\ref{thm:mainDensityThm}, which both simplifies our original proof as well as eliminates one of the assumptions. The authors would also like to thank Shai Revzen for insightful discussions.} This research was supported supported by the Czech 
Science Foundation (GACR) under contract No. 20-11626Y and by the ARO-MURI grant W911NF17-1-0306.

\bibliographystyle{abbrv}
\bibliography{./References}

\end{document}